\documentclass{amsart}
\usepackage{amsmath}
\usepackage{amsfonts}
\usepackage{graphicx}
\usepackage{enumitem}

\usepackage{subfigure}
\usepackage{multirow}
\usepackage{array}
\usepackage{color}

\newtheorem{theorem}{Theorem}

\newtheorem{definition}[theorem]{Definition}

\newtheorem{proposition}[theorem]{Proposition}
\newtheorem{remark}[theorem]{Remark}

\input{tcilatex}

\title[Periodic hydroelastic waves]{Periodic traveling interfacial hydroelastic waves with or without mass}

\author{Benjamin F. Akers}
\address{Department of Mathematics and Statistics,
Air Force Institute of Technology, 2950 Hobson Way, WPAFB, OH 45433 USA}
\email{benjamin.akers@afit.edu}
\thanks{This work was supported in part from a grant from the Office of Naval Research (ONR grant APSHEL to Dr. Akers).}

\author{David M. Ambrose} 
\address{Department of Mathematics, Drexel University, 3141 Chestnut Street, Philadelphia, PA 19104 USA}
\email{dma68@drexel.edu}
\thanks{Dr. Ambrose is grateful to support from the NSF through grant DMS-1515849.}

\author{David W. Sulon}
\address{Department of Mathematics, Drexel University, 3141 Chestnut Street, Philadelphia, PA 19104 USA}
\email{dws57@drexel.edu}

\begin{document}

\begin{abstract}We study the motion of an interface between two irrotational, incompressible fluids,
with elastic bending forces present; this is the hydroelastic wave problem.  We prove a global bifurcation
theorem for the existence of families of spatially periodic traveling waves on infinite depth.  Our traveling
wave formulation uses a parameterized curve, in which the waves are able
to have multi-valued height.  This formulation and the presence of the elastic bending terms allows for the 
application of an abstract global bifurcation theorem of ``identity plus compact'' type.  We furthermore perform
numerical computations of these families of traveling waves, finding that, depending on the choice of parameters,
the curves of traveling waves can either be unbounded, reconnect to trivial solutions, or end with a wave which has a 
self-intersection.  Our analytical and computational methods are  able to treat in a unified way 
the cases of positive or zero mass density
along the sheet, the cases of single-valued or multi-valued height, and the cases of single-fluid or interfacial waves.
\end{abstract}

\maketitle

\section{Introduction}

We study the motion of an elastic, frictionless membrane of non-negative mass between two irrotational,
incompressible fluids.  This is known as a hydroelastic wave problem.
Each fluid has its own non-negative density, and if one of these densities is equal to zero,
this is the hydroelastic water wave case.  Hydroelastic waves can occur in several scenarios, such as a layer of ice
above the ocean \cite{Squire1995} (for which the water wave case would be relevant), or a flapping flag in a fluid \cite{Alben2008} (for which the interfacial case would be relevant).

To model the elastic effects at the free surface, we use the Cosserat theory of elastic shells as developed and described 
by Plotnikov and Toland \cite{plotnikovToland}.  
This system is more suitable for large surface deformations than simpler models
such as linear or Kirchoff-Love models.  The second author, Siegel, and Liu have shown that the initial
value problems for these Cosserat-type hydroelastic waves are well-posed in Sobolev spaces \cite{ambroseSiegelWPHydro},
\cite{LiuAmbroseWellPosednessMass}.  
Toland and Baldi and Toland have proved existence of periodic traveling hydroelastic water
waves with and without mass including studying secondary bifurcations \cite{toland1}, \cite{toland2}, \cite{baldiToland1}, \cite{baldiToland2}.  A number of authors have also computed traveling hydroelastic water waves, finding results
in 2D and 3D, computations of periodic and solitary waves, comparison with weakly nonlinear models, and comparison
across different modelling assumptions for the bending force 
\cite{guyenneParauJFM}, \cite{guyenneParauFS},   \cite{MVBWJFM}, \cite{MVBWPRSA469}, 
\cite{milewskiWangSAM}, \cite{WPMVBPRSA470}, 
 \cite{WVBMIMA}.
While we believe these computations of hydroelastic water waves are the most relevant such studies to the present work, 
this is not an exhaustive list, and the interested reader is encouraged to consult these papers for further references.

We use the formulation for traveling waves introduced by two of the authors and Wright 
\cite{AkersAmbroseWrightTravelingWavesSurfaceTension}.
This version of the traveling wave ansatz is valid for a traveling parameterized curve, and thus extreme behavior
of the waves, such as overturning, may be studied.  Furthermore, while the present study concerns waves in 
two-dimensional fluids, the formulation based on a traveling parameterized curve extends to the case of a
traveling parameterized surface in three space dimensions.  Thus, this method of allowing for overturning waves
generalizes to the higher-dimensional case, unlike methods
based on complex analysis; this has been carried out in one case already \cite{akersReeger}.

In \cite{AkersAmbroseWrightTravelingWavesSurfaceTension},
the density-matched vortex sheet with surface tension was studied.  The particular results  in 
\cite{AkersAmbroseWrightTravelingWavesSurfaceTension} are that the 
formulation for a traveling parameterized curve was introduced and was used to prove a local bifurcation theorem,
and families of waves were computed, showing that curves of traveling waves typically ended when a self-intersecting
wave was reached.  Subsequently, Akers, Ambrose, and Wright showed that Crapper waves, a family of 
exact pure capillary traveling water waves, could be perturbed by including the effect of gravity, and the 
formulation was again used to compute these waves \cite{akersAmbroseWright2}.
Further numerical results were demonstrated in \cite{akersAmbrosePondWright}, where the non-density-matched
vortex sheet was considered.  The formulation was also used to prove a global bifurcation theorem 
for vortex sheets with surface tension for arbitrary constant densities, and thus including the water wave case
\cite{AmbroseStraussWrightGlobBifurc}.
 
We give details of this formulation in Section 2 below, after first stating the
evolution equations for the hydroelastic wave problem.  While the evolution equations, and thus the traveling 
wave equations, are different in the cases with and without mass (i.e., the case of zero mass density or positive 
mass density along the elastic sheet), this difference goes away when applying the abstract global bifurcation
theorem.  This is because the terms involving the mass parameter are nonlinear, and vanish when linearizing
about equilibrium.  We are therefore able to treat the two cases simultaneously in the analysis.

The abstract bifurcation theorem we apply requires a one-dimensional kernel in our linearized operator.
For certain values of the parameters, there may instead be a two-dimensional kernel.  The authors will 
treat the cases of two-dimensional kernels in a subsequent paper.  This will involve studying secondary
bifurcations as in \cite{baldiToland1}, and also studying Wilton ripples \cite{reederShinbrot1}, \cite{reederShinbrot2},
\cite{wilton}, \cite{akers2012wilton}, \cite{trichtchenko2016instability}.

 In Section 3 we state and prove our main theorem, which is a global 
bifurcation theorem for periodic traveling hydroelastic waves, giving several conditions for how a
curve of such waves might end.  In Section 4, we describe our numerical method for computing 
curves of traveling waves, and we give numerical results.

\section{Governing equations}

\subsection{Equations of motion}
The setup of our problem closely mirrors that of \cite{AmbroseStraussWrightGlobBifurc} and \cite{LiuAmbroseWellPosednessMass}.  We consider two two-dimensional fluids, each of which may possess a different mass density.  A one-dimensional interface $I$ (a free surface) comprises the boundary between these two fluid regions; one fluid (with density $\rho_1\geq 0$)  lies below $I$, while the other (with density $\rho_2\geq 0$) lies above $I$.  The fluid regions are infinite in the vertical direction, and are periodic in the horizontal direction.  In our model, we allow the interface itself to possess non-negative mass density $\rho$.  Our model also includes the effects of hydroelasticity and surface tension on the interface; these will be presented later in this section as we introduce the full equations of motion.

Within the interior of each fluid region, the fluid's velocity $u$ is governed by the irrotational, incompressible Euler equations:
\begin{eqnarray*}
u_t + u \cdot \nabla u &=& - \nabla p, \\
\func{div} \left( u \right) &=& 0, \\
u &=& \nabla \phi;
\end{eqnarray*}
however, since $u$ may jump across $I$, there is may still be measure-valued vorticity whose support is $I$.  We can write this vorticity in the form $\gamma \,  \delta_{I}$, where $\gamma \in \mathbb{R}$ (which may vary along $I$) is called the ``unnormalized vortex sheet-strength," and $\delta_{I}$ is the Dirac mass of $I$  \cite{AmbroseStraussWrightGlobBifurc}.

Identifying our overall region with $\mathbb{C}$, we parametrize $I$ as a curve $z\left(
\alpha ,t\right) =x\left( \alpha ,t\right) +iy\left( \alpha ,t\right)$ with periodicity conditions 
\begin{subequations}
\begin{eqnarray}
x\left( \alpha +2\pi ,t\right) &=&x\left( a,t\right) +M,
\label{periodicity_conditions_x} \\
y\left( \alpha +2\pi ,t\right) &=&y\left( \alpha ,t\right),
\label{periodicity_conditions_y}
\end{eqnarray}%
for some $M>0$ (throughout, $\alpha$ will be our spatial parameter along the interface, and $t$ will represent time). \ Let $U$
and $V$ denote the normal and tangential velocities,
respectively; i.e. 
\end{subequations}
\begin{equation}
z_{t}=UN+VT,  \label{position_equation}
\end{equation}%
where 
\begin{eqnarray}
T &=&\frac{z_{\alpha }}{s_{\alpha }}, \notag \\ 
N &=&i\frac{z_{\alpha }}{s_{\alpha }}, \notag \\ 
s_{\alpha }^{2} &=&\left\vert z_{\alpha }\right\vert ^{2}=x_{\alpha
}^{2}+y_{\alpha }^{2}.  \label{s_alpha_definition}
\end{eqnarray}%
(Notice that $T$ and $N$ are the complex versions of the unit tangent and upward normal
vectors to the curve.)  We choose a normalized arclength parametrization;
i.e. we enforce 
\begin{equation}
s_{\alpha }=\sigma \left( t\right) :=\frac{L\left( t\right) }{2\pi }
\label{normalized_arclength_param}
\end{equation}%
at all times $t$, where $L\left( t\right) $ is the length of one period of
the interface. \ Thus, in our parametrization, $s_{\alpha }$ is constant
with respect to $\alpha $. \ Furthermore, we define the tangent angle%
\begin{equation*}
\theta =\arctan \left( \frac{y_{\alpha }}{x_{\alpha }}\right) ;
\end{equation*}%
it is clear that the curve $z$ can be constructed (up to one point) from information about $%
\theta $ and $\sigma $, and that curvature of the interface $\kappa $ can be
given as%
\begin{equation*}
\kappa =\frac{\theta _{\alpha }}{s_{\alpha }}.
\end{equation*}%

The normal velocity $U$ (a geometric invariant) is determined entirely by
the Birkhoff-Rott integral:%
\begin{equation}
U=\func{Re}\left( W^{\ast }N\right) ,  \label{normal_equation}
\end{equation}%
\newline
where%
\begin{equation}
W^{\ast }\left( \alpha ,t\right) =\frac{1}{2\pi i}\func{PV}\int\nolimits_{%
\mathbb{R}
}\frac{\gamma \left( \alpha ^{\prime },t\right) }{z\left( \alpha ,t\right)
-z\left( \alpha ^{\prime },t\right) }d\alpha ^{\prime }.
\label{W_star_definition}
\end{equation}%
We are free to choose the tangential velocity $V$ in order to enforce our
parametrization (\ref{normalized_arclength_param}).  Explicitly, we choose
periodic $V$ such that%
\begin{equation}
V_{\alpha }=\theta _{\alpha }U-\frac{1}{2\pi }\int\nolimits_{0}^{2\pi
}\theta _{\alpha }U\text{ }d\alpha .  \label{V_explicit}
\end{equation}%
We can differentiate both sides of (\ref{s_alpha_definition}) by $t$, and
obtain 
\begin{equation*}
s_{\alpha t}=V_{\alpha }-\theta _{\alpha }U;
\end{equation*}
using this and (\ref{V_explicit}), we can then write 
\begin{equation}\label{s_alpha_t_expression} 
s_{\alpha t} =-\frac{1}{2\pi }\int\nolimits_{0}^{2\pi }\theta _{\alpha }U
\text{ }d\alpha 
=\frac{1}{2\pi }\int\nolimits_{0}^{2\pi }\left( s_{\alpha t}-V_{\alpha
}\right) \text{ }d\alpha
=\frac{1}{2\pi }\int\nolimits_{0}^{2\pi }s_{\alpha t}\text{ }d\alpha. 
\end{equation}
Note that the last step is justified since $V$ is periodic. \ But, we also have
\begin{equation*}
L=\int\nolimits_{0}^{2\pi }s_{\alpha }\text{ }d\alpha ,
\end{equation*}%
\newline
so (\ref{s_alpha_t_expression}) reduces to%
\begin{equation*}
s_{\alpha t}=\frac{L_{t}}{2\pi },
\end{equation*}%
which yields (\ref{normalized_arclength_param}) for all times $%
t$ as long as (\ref{normalized_arclength_param}) holds at $t=0$. \ 

The evolution of the interface is also determined by the behavior of the vortex sheet-strength $\gamma \left( \alpha ,t\right) $,
which can be written in terms of the jump in tangential velocity across the
surface. \ Using a model which combines those used in \cite{ambroseSiegelWPHydro} and \cite{LiuAmbroseWellPosednessMass}, we assume the jump in pressure across the interface to be 
\begin{equation*}
\left[ \left[ p\right] \right] =\rho \left(\func{Re}\left(W_t^\ast N\right) + V_W \theta_t \right) +  \frac{1}{2} E_{b}\left( \kappa _{ss}+\frac{\kappa ^{3}}{2%
}-\tau _{1}\kappa \right) + g \rho \func{Im}N,
\end{equation*}%
where $\rho \geq 0$ is the mass density of the interface, $V_W := V - \func{Re} \left( W^{\ast} T \right)$, $E_{b} \geq 0$ is the bending modulus, $\tau _{1} > 0$ is a surface tension parameter, and $g$ is acceleration due to gravity.  Then, we can write an equation for $\gamma _{t}$ \cite{LiuAmbroseWellPosednessMass}:%
\begin{eqnarray}
\;\;\;\; \gamma _{t} &=&-\frac{\widetilde{S}}{\sigma ^{3}}\left( \partial _{\alpha
}^{4}\theta +\frac{3\theta _{\alpha }^{2}\theta _{\alpha \alpha }}{2}-\tau
_{1}\sigma ^{2}\theta _{\alpha \alpha }\right) +\frac{\left( V_{W}\gamma
\right) _{\alpha }}{\sigma }-2\widetilde{A}\left( \func{Re}\left( W_{\alpha
t}^{\ast }N\right) \right) \label{penultimateGammatEqn}\\
&&-\left( 2A-\frac{2\widetilde{A}\theta _{\alpha }}{%
\sigma }\right) \left( \func{Re}\left( W_{t}^{\ast }T\right) \right)
s_{\alpha } -2\widetilde{A}\left( \left( V_{W}\right) _{\alpha }\theta
_{t}+V_{W}\theta _{t\alpha }+\frac{gx_{\alpha \alpha }}{\sigma }\right) \notag \\
&&-2A\left( \frac{\gamma \gamma _{\alpha }}{4\sigma ^{2}}-V_{W}\func{Re}\left(
W_{\alpha }^{\ast }T\right) +gy_{\alpha }\right). \notag
\end{eqnarray}%
In addition to those defined above, equation (\ref{penultimateGammatEqn}) includes the following constant quantities; some are listed with their physical meanings:
\begin{eqnarray*}
\rho_1 &:& \text{density of the lower fluid } (\, \geq 0), \\
\rho_2 &:& \text{density of the upper fluid } (\, \geq 0), \\
\widetilde{S} &:=&\frac{E_{b}}{\rho _{1}+\rho _{2}}\,  (\, \geq 0), \\
A &:=&\frac{\rho _{1}-\rho _{2}}{\rho _{1}+\rho _{2}} \, \left(\text{the ``Atwood
number,"} \, \in [-1,1]\right), \\
\widetilde{A} &:=&\frac{\rho }{\rho _{1}+\rho _{2}}\, (\, \geq 0).
\end{eqnarray*}%
We can nondimensionalize, and write (\ref{penultimateGammatEqn}) in the
form 
\begin{eqnarray}
\;\;\;\;\;\; \gamma _{t} &=&-\frac{S}{\sigma ^{3}}\left( \partial _{\alpha }^{4}\theta +%
\frac{3\theta _{\alpha }^{2}\theta _{\alpha \alpha }}{2}-\tau _{1}\sigma
^{2}\theta _{\alpha \alpha }\right) +\frac{\left( V_{W}\gamma \right)
_{\alpha }}{\sigma }  \label{tangential_equation} \\
&&-2\widetilde{A}\left( \func{Re}\left( W_{\alpha t}^{\ast }N\right) \right)
-\left( 2A-\frac{2\widetilde{A}\theta _{\alpha }}{\sigma }\right) \left( 
\func{Re}\left( W_{t}^{\ast }T\right) \right) s_{\alpha }  \notag \\
&&-2\widetilde{A}\left( \left( V_{W}\right) _{\alpha }\theta
_{t}+V_{W}\theta _{t\alpha }+\frac{x_{\alpha \alpha }}{\sigma }\right)
-2A\left( \frac{\gamma \gamma _{\alpha }}{4\sigma ^{2}}-V_{W}\func{Re}\left(
W_{\alpha }^{\ast }T\right) +y_{\alpha }\right) ,  \notag
\end{eqnarray}%
where $S = \widetilde{S} \mathbin{/} \left| g \right|$. \ In the two-dimensional hydroelastic vortex sheet problem with
mass, the interface's motion is thus governed by (\ref{position_equation}), (%
\ref{normal_equation}), (\ref{W_star_definition}), (\ref{V_explicit}), and (\ref{tangential_equation}).

\subsection{Traveling wave ansatz\label{reduction_equations}}

We wish to consider traveling wave solutions to the two-dimensional
hydroelastic vortex sheet problem with mass.

\begin{definition}
Suppose $\left( z,\gamma \right) $ is a solution to (\ref{position_equation}%
), (\ref{normal_equation}), (\ref{W_star_definition}), (\ref{V_explicit}), and (\ref{tangential_equation}) that additionally
satisfies $\left( z,\gamma \right) _{t}=\left( c,0\right) $ for some real
parameter $c$. \ We then say $\left( z,\gamma \right) $ is a traveling wave
solution to (\ref{position_equation}), (\ref{normal_equation}), (\ref{W_star_definition}), (\ref{V_explicit}), and (\ref%
{tangential_equation}) with speed $c$.
\end{definition}

\begin{remark}
In our application of global bifurcation theory to show existence of
traveling wave solutions, the value $c$ will serve as our bifurcation
parameter.
\end{remark}

Note that under the traveling wave assumption, we clearly have
\begin{eqnarray*}
U &=&-c\sin \theta, \\
V &=&c\cos \theta.
\end{eqnarray*}%
By carefully differentiating under the integral (in the principal value
sense), it can be shown that under the traveling wave assumption,
both $W_{t}^{\ast }=0$ and $W_{\alpha t}^{\ast }=0$. \ Thus, both terms $2%
\widetilde{A}\left( \func{Re}\left( W_{\alpha t}^{\ast }N\right) \right) $
and $\left( 2A-\frac{2\widetilde{A}\theta _{\alpha }}{\sigma }\right) \left( 
\func{Re}\left( W_{t}^{\ast }T\right) \right) s_{\alpha }$ vanish in the
traveling wave case, and (\ref{tangential_equation}) reduces to 
\begin{eqnarray*}
0 &=&-\frac{S}{\sigma ^{3}}\left( \partial _{\alpha }^{4}\theta +\frac{%
3\theta _{\alpha }^{2}\theta _{\alpha \alpha }}{2}-\tau _{1}\sigma
^{2}\theta _{\alpha \alpha }\right) +\frac{\left( V_{W}\gamma \right)
_{\alpha }}{\sigma } \\
&&-2\widetilde{A}\left( \left( V_{W}\right) _{\alpha }\theta
_{t}+V_{W}\theta _{t\alpha }+\frac{x_{\alpha \alpha }}{\sigma }\right)
-2A\left( \frac{\gamma \gamma _{\alpha }}{4\sigma ^{2}}-V_{W}\func{Re}\left(
W_{\alpha }^{\ast }T\right) +y_{\alpha }\right),
\end{eqnarray*}
or 
\begin{eqnarray}
0 &=&-\frac{S}{\sigma ^{3}}\left( \partial _{\alpha }^{4}\theta +\frac{%
3\theta _{\alpha }^{2}\theta _{\alpha \alpha }}{2}-\tau _{1}\sigma
^{2}\theta _{\alpha \alpha }\right) +\frac{\left( V_{W}\gamma \right)
_{\alpha }}{\sigma }  \label{tan_equation2} \\
&&-2\widetilde{A}\partial _{\alpha }\left( V_{W}\theta _{t}+\frac{x_{\alpha }%
}{\sigma }\right) -2A\left( \frac{\gamma \gamma _{\alpha }}{4\sigma ^{2}}%
-V_{W}\func{Re}\left( W_{\alpha }^{\ast }T\right) +y_{\alpha }\right).  \notag
\end{eqnarray}
Note that 
\begin{equation}\nonumber
V_{W} =V-\func{Re}\left( W^{\ast }T\right) 
=c\cos \theta -\func{Re}\left( W^{\ast }T\right),
\end{equation}
and%
\begin{eqnarray*}
y_{\alpha } &=&\sigma \sin \theta, \\
x_{\alpha } &=&\sigma \cos \theta;
\end{eqnarray*}
also, $\theta _{t}$ clearly vanishes in the traveling wave case. \ We also have%
\begin{equation*}
\left( c\cos \theta -\func{Re}\left( W^{\ast }T\right) \right) \left( \func{%
Re}\left( W_{\alpha }^{\ast }T\right) \right) =-\frac{1}{2}\partial _{\alpha
}\left\{ \left( c\cos \theta -\func{Re}\left( W^{\ast }T\right) \right)
^{2}\right\} ,
\end{equation*}%
\newline
which is shown in \cite{AmbroseStraussWrightGlobBifurc}, or could be computed from the above.  
Thus, we can substitute the
above, and write (\ref{tan_equation2}) as 
\begin{multline}\label{tan_equation_for_numerical}
0 =-\frac{S}{\sigma ^{3}}\left( \partial _{\alpha }^{4}\theta +\frac{%
3\theta _{\alpha }^{2}\theta _{\alpha \alpha }}{2}-\tau _{1}\sigma
^{2}\theta _{\alpha \alpha }\right)   
-2\widetilde{A}\left( \cos \theta \right) _{\alpha }   \\
+\frac{\left( \left( c\cos \theta -\func{Re}\left( W^{\ast }T\right)
\right) \gamma \right) _{\alpha }}{\sigma }  
-A\left( \frac{\partial _{\alpha }\left( \gamma ^{2}\right) }{4\sigma ^{2}}%
+2\sigma \sin \theta +\partial _{\alpha }\left\{ \left( c\cos \theta -\func{%
Re}\left( W^{\ast }T\right) \right) ^{2}\right\} \right).   
\end{multline}
We multiply both sides by $\sigma /\tau _{1}$: 
\begin{eqnarray}
0 &=&-\frac{S}{\tau _{1}\sigma ^{2}}\left( \partial _{\alpha }^{4}\theta +%
\frac{3\theta _{\alpha }^{2}\theta _{\alpha \alpha }}{2}-\tau _{1}\sigma
^{2}\theta _{\alpha \alpha }\right) -\frac{2\widetilde{A}\sigma }{\tau _{1}}%
\left( \cos \theta \right) _{\alpha }  \label{tan_equation3} \\
&&+\frac{1}{\tau _{1}}\left( \left( c\cos \theta -\func{Re}\left( W^{\ast
}T\right) \right) \gamma \right) _{\alpha }\notag \\ 
&&-\frac{A}{\tau _{1}}\left( \frac{%
\partial _{\alpha }\left( \gamma ^{2}\right) }{4\sigma }+2\sigma ^{2}\sin
\theta +\sigma \partial _{\alpha }\left\{ \left( c\cos \theta -\func{Re}%
\left( W^{\ast }T\right) \right) ^{2}\right\} \right) .  \notag
\end{eqnarray}%
For concision, we define 
\begin{multline} \label{Phi_definition} 
\Phi \left( \theta ,\gamma ;c,\sigma \right) :=\frac{1}{\tau _{1}}\left(
\left( c\cos \theta -\func{Re}\left( W^{\ast }T\right) \right) \gamma
\right) _{\alpha } \\
-\frac{A}{\tau _{1}}\left[ \frac{\left( \gamma ^{2}\right) _{\alpha }}{%
4\sigma }+2\sigma ^{2}\sin \theta +\sigma \partial _{\alpha }\left\{ \left(
c\cos \theta -\func{Re}\left( W^{\ast }T\right) \right) ^{2}\right\} \right].
\end{multline}%
(We define $\Phi $ in this manner so that it corresponds with the mapping $\Phi$ defined in
 \cite{AmbroseStraussWrightGlobBifurc}; there, this mapping comprises all of the lower-order terms.) \ Then, (\ref{tan_equation3}) can be written as
\begin{equation*}
0=-\frac{S}{\tau _{1}\sigma ^{2}}\left( \partial _{\alpha }^{4}\theta +\frac{%
3\theta _{\alpha }^{2}\theta _{\alpha \alpha }}{2}-\tau _{1}\sigma
^{2}\theta _{\alpha \alpha }\right) -\frac{2\widetilde{A}\sigma }{\tau _{1}}%
\left( \cos \theta \right) _{\alpha }+\Phi \left( \theta ,\gamma ;c,\sigma
\right) ,
\end{equation*}%
or
\begin{equation}
0=\partial _{\alpha }^{4}\theta +\frac{3\theta _{\alpha }^{2}\theta _{\alpha
\alpha }}{2}-\tau _{1}\sigma ^{2}\theta _{\alpha \alpha }+\frac{2\widetilde{A%
}\sigma ^{3}}{S}\left( \cos \theta \right) _{\alpha }-\frac{\tau _{1}\sigma
^{2}}{S}\Phi \left( \theta ,\gamma ;c,\sigma \right) .  \label{tan_equation4}
\end{equation}%
We label the remaining lower-order terms as%
\begin{eqnarray*}
\Psi _{1}\left( \theta ;\sigma \right) &:=&\frac{3}{2}\theta _{\alpha
}^{2}\theta _{\alpha \alpha }, \\
\Psi _{2}\left( \theta ;\sigma \right) &:=&-\tau _{1}\sigma ^{2}\theta
_{\alpha \alpha }, \\
\Psi _{3}\left( \theta ;\sigma \right) &:=&\frac{2\widetilde{A}\sigma ^{3}}{S%
}\left( \cos \theta \right) _{\alpha };
\end{eqnarray*}%
note that $\Psi _{3}$ is the only remaining term that includes the effect of
interface mass. \ Combining these together as $\Psi :=\Psi _{1}+\Psi _{2}+\Psi _{3},$
we write (\ref{tan_equation4}) as%
\begin{equation}
0=\partial _{\alpha }^{4}\theta +\Psi \left( \theta ;\sigma \right) -\frac{%
\tau _{1}\sigma ^{2}}{S}\Phi \left( \theta ,\gamma ;c,\sigma \right) .
\label{tan_equation5}
\end{equation}

Recall that (\ref{normal_equation}) also determines the behavior of the interface.  Since $U=-c\sin \theta $, (\ref{normal_equation}) becomes (as in \cite{AmbroseStraussWrightGlobBifurc})
\begin{equation}
0=\func{Re}\left( W^{\ast }N\right) +c\cos \theta.
\label{normal_equation_new2}
\end{equation}

Note that (\ref{tan_equation5}) and (\ref{normal_equation_new2}) feature $z$
and $\theta $ interchangeably. \ From this point onward, we would like to
look for traveling wave solutions in the form $\left( \theta ,\gamma \right) 
$ alone; thus, it becomes important to explicitly state how to construct $z$
from $\theta $ (in a unique manner, up to rigid translation) in the
traveling wave case. \ We can clearly do this via
\begin{equation}
z\left( \alpha ,0\right) =z\left( \alpha ,t\right) -ct=\sigma
\int\nolimits_{0}^{\alpha }\exp \left( i\theta \left( \alpha ^{\prime
}\right) \right) \text{ }d\alpha ^{\prime }.  \label{z_construction}
\end{equation}%
Then, given the work completed throughout this section thus far, it is clear
that if $\left( \theta ,\gamma ;c\right) $ satisfy (\ref{tan_equation5}) and
(\ref{normal_equation_new2}) (with $z$ appearing in $\func{Re}\left( W^{\ast
}N\right) $ constructed from $\theta $ via \ref{z_construction}), then $%
\left( z,\gamma \right) $ is a traveling wave solution to (\ref%
{position_equation}), (\ref{normal_equation}) and (\ref{tangential_equation}%
) with speed $c$ (again, with $z$ constructed via (\ref{z_construction})).

It is as this point, however, that we arrive at a technical issue. \ Even if 
$\left( \theta ,\gamma ;c\right) $ yield a traveling wave solution $\left(
z,\gamma \right) $, we cannot expect that $2\pi $-periodic $\theta $ to
yield periodic $z$ via (\ref{z_construction}). \ We would like for any $2\pi $%
-periodic $\left( \theta ,\gamma \right) $ that solve some equations
analogous to (\ref{tan_equation5}) and (\ref{normal_equation_new2}) to
correspond directly to a \textit{periodic} traveling wave solution $\left(
z,\gamma \right) $ of (\ref{position_equation}), (\ref{normal_equation}), and
(\ref{tangential_equation}). \ Hence, in a manner closely analogous to
\cite{AmbroseStraussWrightGlobBifurc}, we modify the mappings in (\ref{tan_equation5})
and (\ref{normal_equation_new2}) to ensure this. \ 

\subsection{Periodicity considerations}

Throught this section, assume that $\theta $ is a sufficiently smooth, $2\pi 
$-periodic function. \ Define the following mean quantities:%
\begin{equation*}
\overline{\cos \theta }:=\frac{1}{2\pi }\int\nolimits_{0}^{2\pi }\cos
\left( \theta \left( \alpha ^{\prime }\right) \right) d\alpha ^{\prime },%
\text{ \ \ \ }\overline{\sin \theta }:=\frac{1}{2\pi }\int\nolimits_{0}^{2%
\pi }\sin \left( \theta \left( \alpha ^{\prime }\right) \right) d\alpha
^{\prime }.
\end{equation*}%
Given $M>0$ and $\theta $ with $\overline{\cos \theta }\neq 0$, define the
``renormalized curve"%
\begin{equation*}
\widetilde{Z}\left[ \theta \right] \left( \alpha \right) :=\frac{M}{2\pi 
\overline{\cos \theta }}\left[ \int\nolimits_{0}^{\alpha }\exp \left(
i\theta \left( \alpha ^{\prime }\right) \right) \text{ }d\alpha ^{\prime
}-i\alpha \overline{\sin \theta }\right] .
\end{equation*}%
Note that $\widetilde{Z}\left[ \theta \right] $ is one derivative smoother
than $\theta $, and a direct calculation shows that such a curve in fact
satisfies our original spacial periodicity requirement (\ref%
{periodicity_conditions_x}), (\ref{periodicity_conditions_y}):%
\begin{equation*}
\widetilde{Z}\left[ \theta \right] \left( \alpha +2\pi \right) =\widetilde{Z}%
\left[ \theta \right] \left( \alpha \right) +M.
\end{equation*}%
Also, we clearly have normal and tangent vectors to $\widetilde{Z}\left[ \theta %
\right] $ given by 
\begin{eqnarray*}
\widetilde{T}\left[ \theta \right] &=&\frac{\partial _{\alpha }\widetilde{Z}%
\left[ \theta \right] }{\left\vert \partial _{\alpha }\widetilde{Z}\left[
\theta \right] \right\vert }, \\
\widetilde{N}\left[ \theta \right] &=&i\frac{\partial _{\alpha }\widetilde{Z}%
\left[ \theta \right] }{\left\vert \partial _{\alpha }\widetilde{Z}\left[
\theta \right] \right\vert }.
\end{eqnarray*}%

Next, we use a specific form of the Birkhoff-Rott integral (for real-valued $%
\gamma $ and complex-valued $\omega $ that satisfy $\omega \left( \alpha
+2\pi \right) =\omega \left( \alpha \right) +M$):%
\begin{equation*}
B\left[ \omega \right] \gamma \left( \alpha \right) :=\frac{1}{2iM}\func{PV}%
\int\nolimits_{%
\mathbb{R}
}\gamma \left( \alpha ^{\prime }\right) \cot \left( \frac{\pi }{M}\left(
\omega \left( \alpha \right) -\omega \left( \alpha ^{\prime }\right) \right)
\right) d\alpha ^{\prime }.
\end{equation*}%
As discussed in \cite{AmbroseStraussWrightGlobBifurc}, setting $\omega =z$ yields $%
W^{\ast }=B\left[ z\right] \gamma $, where $W^{\ast }$ is as defined in (\ref%
{W_star_definition}); this follows from the well-known cotangent series expansion due to Mittag-Leffler (which can, for example, be found in \cite{AblowitzFokasComplexVar}). \ We are now ready to define a mapping $\widetilde{%
\Phi }$, analogous to the mapping in \cite{AmbroseStraussWrightGlobBifurc}:%
\begin{eqnarray*}
\widetilde{\Phi }\left( \theta ,\gamma ;c\right)
&:=&\frac{1}{\tau _{1}}\partial _{\alpha }\left\{ c\cos \theta -\func{Re}%
\left( \left(B\left[ \widetilde{Z}\left[ \theta \right] \right] \gamma\right)  
\widetilde{T}\left[ \theta \right] \right) \gamma \right\} \\
&&-\frac{A}{\tau _{1}}\left( \frac{\pi \overline{\cos \theta }}{2M}\partial
_{\alpha }\left( \gamma ^{2}\right) +\frac{M^{2}}{2\pi ^{2}\left( \overline{%
\cos \theta }\right) ^{2}}\left( \sin \theta -\overline{\sin \theta }\right)
\right) \\
&&-\frac{A}{\tau _{1}}\left( \frac{M}{2\pi \overline{\cos \theta }}\partial
_{\alpha }\left\{ \left( c\cos \theta -\func{Re}\left( \left(B\left[ \widetilde{Z}%
\left[ \theta \right] \right] \gamma\right) \widetilde{T}%
\left[ \theta \right] \right) \right) ^{2}\right\} \right) .
\end{eqnarray*}

This construction is enough for us to ensure $M$-periodicity in a
traveling-wave wave solution to (\ref{tan_equation5}) and (\ref%
{normal_equation_new2}):

\begin{proposition}
\label{spacial_periodic_version}Suppose $c\neq 0$ and $2\pi $-periodic
functions $\theta ,\gamma $ satisfy $\overline{\cos \theta }\neq 0$ and%
\begin{equation}
\func{Re}\left( \left(B\left[ \widetilde{Z}\left[ \theta \right] \right] \gamma\right)
 \widetilde{N}\left[ \theta \right] \right) +c\sin
\theta =0,  \label{normal_equation_spacial2}
\end{equation}
\begin{equation}
\partial _{\alpha }^{4}\theta +\Psi \left( \theta ;\sigma \right) -\frac{%
\tau _{1}\sigma ^{2}}{S}\widetilde{\Phi }\left( \theta ,\gamma ;c\right) =0,
\label{tan_equation_spacial2}
\end{equation}%
with $\sigma =M/\left( 2\pi \overline{\cos \theta }\right) $. \ Then, $%
\left( \widetilde{Z}\left[ \theta \right] \left( \alpha \right) +ct,\gamma
\left( \alpha \right) \right) $ is a traveling wave solution to (\ref%
{tan_equation5}) and (\ref{normal_equation_new2}) with speed $c$, and $%
\widetilde{Z}\left[ \theta \right] \left( \alpha \right) +ct$ is spatially
periodic with period $M$. \ 
\end{proposition}

A proof of a proposition almost identical to Proposition \ref{spacial_periodic_version} can be found in \cite{AmbroseStraussWrightGlobBifurc}. \ Under the assumptions of
this proposition, we can see how (\ref{normal_equation_spacial2})
corresponds to (\ref{normal_equation_new2}) given our construction above;
then, \cite{AmbroseStraussWrightGlobBifurc} shows that (under these assumptions) $%
\widetilde{\Phi }\left( \theta ,\gamma ;c\right) =\Phi \left( \theta ,\gamma
;c,\sigma \right) $ with $\sigma =M/\left( 2\pi \overline{\cos \theta }%
\right) $. \ 

We thus will henceforth work with equations (\ref{normal_equation_spacial2}), (\ref%
{tan_equation_spacial2}), though a few more steps are needed in order to
bring these equations into a form conducive to applying the global
bifurcation theory.

\subsection{Final reformulation} \label{subseq:finalform}

We wish to ``solve" (\ref{tan_equation_spacial2}) for $\theta $. \ To do so,
we introduce an inverse derivative operator $\partial _{\alpha }^{-4}$, which we define in Fourier space.  For a general $2\pi $%
-periodic map $\mu $ with convergent Fourier series, let $\widehat{\mu }%
\left( k\right) $ denote the $k^{\text{th}}$ Fourier coefficient in the usual
sense, i.e. $\mu \left( \alpha \right) =\sum\nolimits_{k=-\infty }^{\infty }%
\widehat{\mu }\left( k\right) \exp \left( ik\alpha \right) $. \ Then, define for $\mu$ with mean zero (i.e. $\widehat{\mu} (0) = 0$)
\begin{equation}
\label{eq:inverseDerivDef}
\begin{cases}
\widehat{\partial _{\alpha }^{-4}\mu }\left( k\right) :=k^{-4} \widehat{\mu }\left( k\right), & k \neq 0 \\
             \widehat{\partial _{\alpha }^{-4}\mu }\left( 0\right) := 0. 
\end{cases}
\end{equation}
By this construction, $\partial _{\alpha }^{-4}$ clearly preserves periodicity, and for sufficiently regular, periodic $\mu$ with mean zero, 
\begin{equation*}
\partial _{\alpha }^{-4}\partial _{\alpha }^{4}\mu= \mu = \partial _{\alpha }^{4}\partial _{\alpha }^{-4}\mu.
\end{equation*}
Also, define the projection $P$ (here, $\mu$ may not necessarily have mean zero):
\begin{equation}
\label{eq:projDef}
P\mu (\alpha) :=\mu \left(\alpha \right)-\frac{1}{2\pi }%
\int\nolimits_{0}^{2\pi } \mu \left( \alpha^\prime \right) d\alpha^\prime;
\end{equation}
it is clear that $P\mu$ has mean zero. \ We apply $\partial _{\alpha }^{-4} P$ to both sides of (\ref{tan_equation_spacial2}), and obtain the equation
\begin{equation}
0=\theta +\partial _{\alpha }^{-4} P \Psi \left( \theta ;\sigma \right)
-\frac{\tau _{1}\sigma ^{2}}{S}\partial _{\alpha }^{-4} P \widetilde{\Phi }%
\left( \theta ,\gamma ;c\right)  \label{tan_equation6}
\end{equation}%
(throughout, note $\sigma =M/\left( 2\pi \overline{\cos \theta }\right) $).

Next, we approach (\ref{normal_equation_spacial2}). \ First, we subtract the
mean $\overline{\gamma }:=\left( 2\pi \right) ^{-1}\int\nolimits_{0}^{2\pi
}\gamma \left( \alpha \right) $ $d\alpha $ from $\gamma $ and write%
\begin{equation*}
\gamma _{1}:=\gamma -\overline{\gamma },
\end{equation*}%
As in \cite{AmbroseStraussWrightGlobBifurc}, the Birkhoff-Rott integral can be decomposed as 
\begin{equation}
B\left[ \omega \right] \gamma \left( \alpha \right) = \frac{1}{2 i \, \omega_{\alpha} (\alpha)} H \gamma (\alpha) + \mathcal{K}\left[\omega\right] \gamma (\alpha),
\end{equation}
where the most singular portion
\begin{equation} \label{eq:hilbertTransDef}
H \gamma (\alpha) := \frac{1}{2 \pi}\, \func{PV}\int \nolimits_0^{2 \pi} \gamma \left( \alpha^\prime \right) \cot \left(\frac{1}{2}\left( \alpha - \alpha^\prime \right) \right) d \alpha^\prime
\end{equation}
is the Hilbert transform, and the remainder
\begin{eqnarray*}
&&\mathcal{K} \left[ \omega \right] \gamma \left( \alpha \right) :=  \\
&&\frac{1}{4\pi i}\func{PV}\int\nolimits_{0}^{2\pi }\gamma \left( \alpha
^{\prime }\right) \left[ \cot \left( \frac{1}{2}\left( \omega \left( \alpha
\right) -\omega \left( a^{\prime }\right) \right) \right) -\frac{1}{\partial
_{\alpha ^{\prime }}\omega \left( \alpha ^{\prime }\right) }\cot \left( \frac{1}{2}%
\left( \alpha -\alpha ^{\prime }\right) \right) \right] d\alpha ^{\prime }
\end{eqnarray*}%
is smooth on the domain we later define in Section \ref{subsubseq:mappingProperties}.  Then, as is also done in \cite{AmbroseStraussWrightGlobBifurc}, we write (\ref{normal_equation_spacial2}) in
the form%
\begin{equation}
\gamma _{1}-H\left\{ 2\left\vert \partial _{\alpha }\widetilde{Z}\left[
\theta \right] \right\vert \func{Re}\left( \left( \mathcal{K} \left[ \widetilde{Z}\left[
\theta \right] \right] \left( \overline{\gamma }+\gamma _{1}\right) \right) 
\widetilde{N}\left[ \theta \right] \right) +2c\left\vert \partial _{\alpha }%
\widetilde{Z}\left[ \theta \right] \right\vert \sin \theta \right\} =0.
\label{nor_equation2}
\end{equation}%

Define%
\begin{equation*}
\Theta \left( \theta ,\gamma _{1};c\right) :=\frac{\tau _{1}\sigma ^{2}}{S}%
\partial _{\alpha }^{-4}P \widetilde{\Phi }\left( \theta ,\gamma _{1}+%
\overline{\gamma };c\right) -\partial _{\alpha }^{-4} P \Psi \left( \theta
;\sigma \right) ,
\end{equation*}%
so (\ref{tan_equation6}) becomes%
\begin{equation}
\theta -\Theta \left( \theta ,\gamma _{1};c\right) =0.
\label{theta_equation_new}
\end{equation}%
We then substitute $\theta =\Theta \left( \theta ,\gamma _{1};c\right) $
into (\ref{nor_equation2}) to obtain%
\begin{equation}
\gamma _{1}-\Gamma \left( \theta ,\gamma _{1};c\right) =0,
\label{gamma_equation_new}
\end{equation}%
where%
\begin{eqnarray}\label{gamma_equation_new2}
&&\Gamma \left( \theta ,\gamma _{1};c\right) := \\
&&H\{2\left\vert \partial
_{\alpha }\widetilde{Z}\left[ \Theta \left( \theta ,\gamma _{1};c\right) %
\right] \right\vert \func{Re}\left( \left( \mathcal{K} \left[ \widetilde{Z}\left[
\Theta \left( \theta ,\gamma _{1};c\right) \right] \right] \left( \overline{%
\gamma }+\gamma _{1}\right) \right) \widetilde{N}\left[ \Theta \left( \theta
,\gamma _{1};c\right) \right] \right)\notag \\ 
&&+2c\left\vert \partial _{\alpha }\widetilde{Z}\left[ \Theta \left( \theta
,\gamma _{1};c\right) \right] \right\vert \sin \left( \Theta \left( \theta
,\gamma _{1};c\right) \right) \}. \notag
\end{eqnarray}

In Section \ref{subsubseq:mappingProperties}, we will show compactness of $\left( \Theta ,\Gamma \right) $ given
appropriate choice of domain, as the bifurcation theorem we shall apply to (%
\ref{theta_equation_new}), (\ref{gamma_equation_new}) requires an ``identity
plus compact" formulation.

\section{Global bifurcation theorem}

\subsection{Main theorem}

We now present a global bifurcation theorem for the traveling wave problem.
\ In essence, this theorem shows the existence of a rich variety of
nontrivial solution sets. \ We prove this theorem throughout this section.

\begin{theorem}
\label{main_theorem}Let all be as defined in the previous section. \ For all
choices of constants $M>0$, $S>0$, $\tau _{1}>0$, $A\in \left[ -1,1\right] $%
, $\widetilde{A}\geq 0$, $\overline{\gamma }\in 
\mathbb{R}
$, there exists a countable number of connected sets of smooth, non-trivial
traveling-wave solutions of the two-dimensional hydroelastic vortex sheet
problem with mass (i.e. solutions to $\left( \theta -\Theta \left( \theta
,\gamma _{1};c\right) ,\gamma _{1}-\Gamma \left( \theta ,\gamma
_{1};c\right) \right) =\left( 0,0;c\right) $). \ If $\overline{\gamma }\neq
0 $ or $A\neq 0$, then each of these connected sets have at least one of the
following properties (a) -- (e): \\

\begin{enumerate}[label={\upshape(\alph*):}]

\item It contains waves with arbitrarily long interface lengths per period

\item It contains waves whose interfaces have curvature with arbitrarily large derivative

\item It contains waves in which the derivative of the jump of the tangential component of fluid velocity can be arbitrarily large

\item Its closure contains a wave whose interface self-intersects

\item It contains a sequence of waves whose interface converge to a trivial solution but whose speeds contain at least two convergent subsequences whose limits differ. \\
\end{enumerate}
\noindent If $\overline{\gamma }=0$ and $A=0$,
then another possible outcome is \\
\begin{enumerate}[label={\upshape(\alph*):},resume]

\item It contains waves which have speeds which are arbitrarily large. \\

\end{enumerate}
\end{theorem}

\begin{remark}
The possible outcomes listed in the above theorem are very similar to the
analogous main theorem of \cite{AmbroseStraussWrightGlobBifurc}. \ Notably different is
outcome (b), where we list the possibility for the derivative of curvature
to arbitrarily grow (instead of merely curvature itself). \ This distinction
arises from a difference in domain spaces used; here, we require $\theta $ to
possess one higher derivative than in \cite{AmbroseStraussWrightGlobBifurc}.
\end{remark}

\subsection{Global bifurcation results}

\subsubsection{General global bifurcation theory}

Our main theorem posits the existence of certain solution sets to the
traveling wave problem; we show that this essentially follows directly from
an application of a global bifurcation theorem due to Rabinowitz \cite{RabinowitzNonlinearEig} and generalized by Kielh\"{o}fer
\cite{KielhoferBifurcationBook}. \ The conditions of the theorem require a notion of 
\textit{odd crossing number} for families of bounded linear operators.

\begin{definition}
\label{odd_crossing_number_def}Assume $A\left( c\right) $ is a family of
bounded linear operators depending continuously on a real parameter $c$. \
Suppose at some $c=c_{0}$, $A\left( c\right) $ has a zero eigenvalue. \
Define $\sigma ^{<}\left( c\right) :=1$ if there are no negative real
eigenvalues of $A\left( c\right) $ that perturb from this zero eigenvalue of 
$A\left( c_{0}\right) $, and $\sigma ^{<}\left( c\right) :=\left( -1\right)
^{m_{1}+\dots +m_{k}}$ if $\mu _{1},\dots ,\mu _{k}$ are all negative real
eigenvalues of $A\left( c\right) $ that perturb from this zero eigenvalue of 
$A\left( c_{0}\right) $, each with algebraic multiplicities $m_{1},\dots
,m_{k}$. \ We say $A\left( c\right) $ has an odd crossing number at $c=c_{0}$
if (i) $A\left( c\right) $ is regular for $c\in \left( c_{0}-\delta
,c_{0}\right) \cup \left( c_{0},c_{0}+\delta \right) $ and (ii) $\sigma
^{<}\left( c\right) $ changes sign at $c=c_{0}$.
\end{definition}

\begin{remark}
We can think of the crossing number itself as the number of real eigenvalues
(counted with algebraic multiplicity) of $A\left( c\right) $ that pass
through $0$ as $c$ moves across $c_{0}$ \cite{AmbroseStraussWrightGlobBifurc}.
\end{remark}

We now state the abstract theorem as it is appears in \cite{AmbroseStraussWrightGlobBifurc}, which itself is a slight modification of the Kielh\"{o}fer version
(see Remark \ref{rem:abstractThmBoundary} below).

\begin{theorem}
\label{general_bifurcation_theorem}(General bifurcation theorem). \ Let $X$
be a Banach space, and let $U$ be an open subset of $X\times 
\mathbb{R}
$. \ Let $F$ map $U$ continuously into $X$. \ Assume that \\
\begin{enumerate}[label={\upshape(\alph*):}]
\item the Frechet derivative $D_{\xi }F\left( 0,\cdot \right) $ belongs to $C\left( 
\mathbb{R}
,L\left( X,X\right) \right) $

\item the mapping $\left( \xi ,c\right) \mapsto F\left( \xi
,c\right) -\xi $\ is compact from $X\times 
\mathbb{R}
$ into $X$\textit{, and}

\item $F\left( 0,c_{0}\right) =0$ and $D_{x}F\left(
0,c\right) $ has an odd crossing number at $c=c_{0}$. \\
\end{enumerate}
Let $S$ denote the closure of the set of nontrivial
solutions of $F\left( \xi ,c\right) =0$ in $X\times 
\mathbb{R}
$. \ Let $C$ denote the connected component of $S$ to which $\left( 0,c_{0}\right) $ belongs. \ Then, one of the
following alternatives is valid: \\

\begin{enumerate}[label={\upshape(\roman*):}]
\item $C$ is unbounded; or

\item $C$ contains a point $\left( 0,c_{1}\right) $ where $c_{0}\neq c_{1}$ ; or

\item $C$ contains a point on the boundary of $U$. \\
\end{enumerate}
\end{theorem}

\begin{remark} \label{rem:abstractThmBoundary}
The Kielh\"{o}fer version of the theorem explicitly assumes the case $%
U=X\times 
\mathbb{R}
$. \ The proof, however, is easily modified to admit general open $%
U\subseteq X\times 
\mathbb{R}
$. \ The choice of such $U$ for our problem will ensure well-definedness and
compactness of our mapping $\left( \Theta ,\Gamma \right) .$
\end{remark}

One condition for Theorem \ref{general_bifurcation_theorem} is that the mapping in
question can be written in the form ``identity plus compact." \ We show that $%
\left( \Theta ,\Gamma \right) $ is in fact compact over an appropriately
chosen domain.

\subsubsection{Mapping properties} \label{subsubseq:mappingProperties}

To begin, we set up the necessary notation for the function spaces we wish
to work with.

\begin{definition}
Let $H_{\text{per}}^{s}$ denote $H_{\text{per}}^{s}\left[ 0,2\pi \right] ,$
i.e. the usual Sobelev space of $2\pi $-periodic functions from $%
\mathbb{R}
$ to $%
\mathbb{C}
$ with square-integrable derivatives up to order $s\in 
\mathbb{N}
$. \ Let $H_{\text{per,odd}}^{s}$ denote the subset of $H_{\text{per}}^{s}$
comprised of odd functions; define $H_{\text{per,even}}^{s}$ similarly. \
Let $H_{\text{per,}0\text{,even}}^{s}$ denote the subset of $H_{\text{%
per,even}}^{s}$ comprised of mean-zero functions. \ Finally, letting $H_{%
\text{loc}}^{s}$ denote the usual Sobolev space of functions in $H^{s}\left(
I\right) $ for all bounded intervals $I$, we put%
\begin{equation*}
H_{M}^{s}=\left\{ \omega \in H_{\text{loc}}^{s}:\omega \left( \alpha \right)
-\frac{M\alpha }{2\pi }\in H_{\text{per}}^{s}\right\} .
\end{equation*}
For $b\geq 0$ and $s\geq 2$, define the ``chord-arc space"%
\begin{equation*}
C_{b}^{s}=\left\{ \omega \in H_{M}^{s}:\inf_{\alpha ,\alpha ^{\prime }\in 
\left[ a,b\right] }\left\vert \frac{\omega \left( \alpha \right) -\omega
\left( \alpha ^{\prime }\right) }{\alpha ^{\prime }-\alpha }\right\vert
>b\right\} .
\end{equation*}%
\end{definition}

We are now ready to set an appropriate domain for $\left( \Theta ,\Gamma
\right) $, and assert its compactness as a mapping over such domain.

\begin{proposition}
\label{mapping_proposition}Put%
\begin{equation*}
X=H_{\text{per},\text{odd}}^{2}\times H_{\text{per},0,\text{even}}^{1}\times 
\mathbb{R}%
\end{equation*}%
\newline
and%
\begin{equation}
U_{b,h}=\left\{ \left( \theta ,\gamma _{1};c\right) \in X
:\overline{\cos \theta }>h,\widetilde{Z}\left[ \theta \right] \in C_{b}^{2}%
\text{ and } \widetilde{Z}\left[ \Theta \left( \theta ,\gamma _{1};c\right) %
\right] \in C_{b}^{5}\right\} .  \label{Ubh_definition}
\end{equation}%
\newline
The mapping $\left( \Theta ,\Gamma \right) $ (where $\Theta ,\Gamma $ are as
defined in Section \ref{subseq:finalform}) from $U_{b,h}\subseteq X$ into 
$X$ is compact.
\end{proposition}

\begin{proof}
The chord-arc conditions are imposed to ensure the well-definedness of the
Birkhoff-Rott integral; this is seen in \cite{AmbroseStraussWrightGlobBifurc}. \ With
these conditions, alongside the condition $\overline{\cos \theta }>h$, we
ensure that each component of $\widetilde{\Phi }$ is well-defined over $%
U_{b,h}$. \ It is also demonstrated in \cite{AmbroseStraussWrightGlobBifurc} that $%
\widetilde{\Phi }$ costs a derivative in $\theta $ and retains derivatives
in $\gamma $; here, we have that $\widetilde{\Phi }$ maps from $U_{b,h}$ to $%
H_{\text{per},0}^{1}$. \ We need to check the mapping properties of $\Psi $ (i.e.,
the terms that differ from the analogous equation of \cite{AmbroseStraussWrightGlobBifurc}). \ Recall that $\Psi :=\Psi _{1}+\Psi _{2}+\Psi _{3}$, where%
\begin{eqnarray*}
\Psi _{1}\left( \theta ;\sigma \right) &=&\frac{3}{2}\theta _{\alpha
}^{2}\theta _{\alpha \alpha }=\frac{1}{2}\partial _{\alpha }\left[ \theta
_{\alpha }^{3}\right] \\
\Psi _{2}\left( \theta ;\sigma \right) &=&-\tau _{1}\sigma ^{2}\theta
_{\alpha \alpha } \\
\Psi _{3}\left( \theta ;\sigma \right) &=&\frac{2\widetilde{A}\sigma ^{3}}{S}%
\left( \cos \theta \right) _{\alpha }
\end{eqnarray*}%

Using elementary results regarding algebra properties for Sobolev
spaces, we have that the maps $\left( \cdot \right) ^{3}$ and $\cos \left(
\cdot \right) $ both map from $H_{\text{per}}^{s}$ to $H_{\text{per}}^{s}$
as long as $s>\frac{n}{2}=\frac{1}{2}$. \ The choice $s=2$ satisfies this. \
Also, if $\theta $ is odd, $\theta _{\alpha }$ is even (as is $\theta
_{\alpha }^{3}$), so $\partial _{\alpha }\left[ \theta _{\alpha }^{3}\right] 
$ is odd. \ The function $\theta _{\alpha \alpha }$ is also odd, so $\Psi _{1},\Psi _{2}$
maps into an ``odd" space. \ Moreover, since $\partial _{\alpha }\cos \left(
\theta \right) =-\left( \sin \theta \right) \left( \partial _{\alpha }\theta
\right) $ and $\partial _{\alpha }\theta $ is even, we see that $\Psi _{3}$
maps into an "odd" space as well. \ Thus, we can write 
\begin{equation*}
\Psi :H_{\text{per},\text{odd}}^{2}\rightarrow H_{\text{per},\text{odd}}^{0},
\end{equation*}%
so $\partial _{\alpha }^{-4} P \Psi $ maps into $H_{\text{per}}^{4}$. \ 

Furthermore, $\Psi $ maps
bounded sets to bounded sets as well, as each $\partial _{\alpha }$ is a
bounded linear map between appropriate Sobolev spaces, and $\left( \cdot \right) ^{3}$ also maps bounded sets to bounded sets given that its
domain satisfies the condition $s>\frac{1}{2}$. \ Since $\partial _{\alpha
}^{-4}$ and $P$ are also bounded linear maps, we have that $\partial _{\alpha
}^{-4} P \Psi $ maps bounded sets to bounded sets. \ Also, it is clear that $%
\partial _{\alpha }^{-4} P \Psi \left( \theta ;\sigma \right) $ retains parity
of $\theta $. \ 

We summarize the mapping properties as follows:%
\begin{eqnarray*}
\widetilde{\Phi } &:&U_{b,h}\rightarrow H_{per,0,odd}^{1}, \\
\partial _{\alpha }^{-4} P \widetilde{\Phi } &:&U_{b,h}\rightarrow
H_{\text{per},\text{odd}}^{5}, \\
\partial _{\alpha }^{-4} P \Psi &:&U_{b,h}\rightarrow H_{\text{per},\text{odd}}^{4},
\end{eqnarray*}%
so by the above, we have by our definition of $\Theta,$%
\begin{equation*}
\Theta :U_{b,h}\rightarrow H_{\text{per},\text{odd}}^{4}.
\end{equation*}%
As in \cite{AmbroseStraussWrightGlobBifurc}, $\Gamma $ is written as a composition of $%
\Theta $ and an operator that neither gains nor costs derivatives, so%
\begin{equation*}
\Gamma :U_{b,h}\rightarrow H_{\text{per},0,\text{even}}^{4}\text{.}
\end{equation*}%
By Rellich's theorem, bounded sets in $H_{\text{per},\text{odd}}^{4}$ are precompact in $%
H_{\text{per},\text{odd}}^{2}$, and bounded sets $H_{\text{per},0,\text{even}}^{4}$ are precompact in $%
H_{\text{per},0,\text{even}}^{1}$. \ Each term of $\left( \Theta ,\Gamma \right) $ maps
bounded sets to bounded sets. \ Thus, by viewing $\left( \Theta ,\Gamma
\right) $ as a map from open $U_{b,h}\subseteq X$ into $X$, we have that $%
\left( \Theta ,\Gamma \right) $ is a compact map.
\end{proof}

We next compute the Frech\'{e}t derivative of $\left( \theta -\Theta \left(
\theta ,\gamma _{1};c\right) ,\gamma _{1}-\Gamma \left( \theta ,\gamma
_{1};c\right) \right) $, and subsequently use this to analyze the crossing
number.

\subsubsection{Linearization calculation}

In order to abbreviate the linearization calculations of $\left( \Theta
,\Gamma \right) $, we introduce some notation. \ For any map $\mu \left(
\theta ,\gamma _{1};c\right) $, let $\left( \overrightarrow{\theta },%
\overrightarrow{\gamma }\right) $ denote the direction of differentiation,
and define for general, sufficiently regular mappings $\mu$
\begin{eqnarray*}
\mu _{0} &:=&\mu \left( 0,0;c\right) \\
D\mu &:=&\left. D_{\theta ,\gamma _{1}}\mu \left( \theta ,\gamma
_{1};c\right) \right\vert _{\left( 0,0;c\right) } \left( \overrightarrow{\theta },%
\overrightarrow{\gamma }\right):=\lim_{\varepsilon \rightarrow 0}\frac{1}{\varepsilon }%
\left( \mu \left( \varepsilon \overrightarrow{\theta },\varepsilon 
\overrightarrow{\gamma };c\right) -\mu _{0}\right) .
\end{eqnarray*}%
Note that $\sigma $ is dependent on $\theta $; we denote $\Sigma \left(
\theta \right) :=\sigma =M/\left( 2\pi \overline{\cos \theta }\right) $. \
We note the following elementary results:%
\begin{eqnarray*}
\sin _{0} &=&0,\text{ \ \ \ }D\cos =0 \\
\Sigma _{0} &=&\frac{M}{2\pi },\text{ \ \ }D\Sigma =0.
\end{eqnarray*}

A large number of components of $\left( D\Theta ,D\Gamma \right) $ were
explicitly computed in \cite{AmbroseStraussWrightGlobBifurc}. \ Namely, these results
yield for our closely analogous $\widetilde{\Phi }$ 
\begin{equation*}
D \widetilde{\Phi }=-\frac{\pi \overline{\gamma }}{M}\left( -\frac{\overline{%
\gamma }}{\tau _{1}}+\frac{cAM}{\pi \tau _{1}}\right) \partial _{\alpha }H%
\overrightarrow{\theta }-\frac{AM^{2}}{2\pi ^{2}\tau _{1}}P\overrightarrow{%
\theta }+\left( \frac{c}{\tau _{1}}-\frac{\pi A\overline{\gamma }}{\tau _{1}M%
}\right) \partial _{\alpha }\overrightarrow{\gamma },
\end{equation*}
where the projection $P$ is as defined in (\ref{eq:projDef}) and $H$ is the Hilbert transform (see (\ref{eq:hilbertTransDef})). \ We need to compute the linearization of
the ``extra" terms $\Psi _{1}, \Psi _{2}, \Psi _{3}$ that (loosely) correspond
to the hydroelastic and interface mass effects. \ Examine
\begin{eqnarray*}
D\Psi _{1} &=&\frac{1}{2}\partial _{\alpha }D\left[ \theta _{\alpha }^{3}%
\right] \\
&=&\frac{3}{2}\partial _{\alpha }\left[ \left( \partial _{\alpha }\theta
\right) ^{2}\partial _{\alpha }\overrightarrow{\theta }\right] ,
\end{eqnarray*}
so at $\theta =0$, we see $D\Psi _{1}=0$. \ Next, we examine
\begin{eqnarray*}
D\Psi _{2} &=&D\left[ -\tau _{1}\Sigma ^{2}\partial _{\alpha }^{2}\theta %
\right] \\
&=&-\tau _{1}\left( 2\Sigma _{0}D\Sigma \left[ \partial _{\alpha }^{2}\theta
\right]_{\theta =0}+\Sigma _{0}^{2}\partial _{\alpha }^{2}D\theta
\right) \\
&=&-\tau _{1}\left( 0+\left( \frac{M}{2\pi }\right) ^{2}\partial _{\alpha
}^{2}\overrightarrow{\theta }\right) \\
&=&-\tau _{1}\left( \frac{M}{2\pi }\right) ^{2}\partial _{\alpha }^{2}%
\overrightarrow{\theta },
\end{eqnarray*}
and
\begin{eqnarray*}
D\Psi _{3} &=&D\left[ -\frac{2\widetilde{A}g\Sigma ^{3}}{S}\sin \left(
\theta \right) \theta _{\alpha }\right] \\
&=&-\frac{2\widetilde{A}}{S}\left[ D\left( \Sigma ^{3}\right) \sin
_{0}\left[ \partial _{\alpha }\left( \theta \right) \right]_{\theta
=0}+\Sigma _{0}^{3}D\cos \left[\partial _{\alpha }\left( \theta \right)
\right]_{\theta =0}+\Sigma _{0}^{3}\sin _{0}\partial _{\alpha }%
\overrightarrow{\theta }\right] \\
&=&-\frac{2\widetilde{A}}{S}\left[ 0+0+0\right] \\
&=&0.
\end{eqnarray*}
Thus, 
\begin{equation*}
D\Psi =-\tau _{1}\left( \frac{M}{2\pi }\right) ^{2}\partial _{\alpha }^{2}%
\overrightarrow{\theta }.
\end{equation*}

We pause to remark that the crossing number is entirely determined by the linearization
near equilibrium. \ Note that $D\Psi _{3}=0$ means that the presence of
interface mass will not have any bearing on the application of Theorem \ref%
{general_bifurcation_theorem} to our problem; in other words, the same
conclusions about bifurcation (given odd crossing number) can be drawn in the $%
\widetilde{A}=0$ as in the $\widetilde{A}>0$ case. \ 

Continuing, we recall the definition of $\Theta$, and calculate%
\begin{eqnarray*}
D\Theta &=&\partial _{\alpha }^{-4} P \left[ \frac{\tau _{1}}{S}D\left( \Sigma
^{2}\widetilde{\Phi }\right) -D\Psi \right] \\
&=&\partial _{\alpha }^{-4} P \left[ \frac{\tau _{1}}{S}\left( 2\Sigma
_{0}D\Sigma \widetilde{\Phi }_{0}+\Sigma _{0}^{2}D\widetilde{\Phi }\right)
-D\Psi \right] \\
&=&\partial _{\alpha }^{-4} P \left[ \frac{\tau _{1}}{S}\left( \frac{M}{2\pi }%
\right) ^{2}D\widetilde{\Phi }-D\Psi \right] \\
&=&\frac{\tau _{1}\overline{\gamma }M}{4\pi S}\left( \frac{\overline{\gamma }%
}{\tau _{1}}-\frac{cAM}{\pi \tau _{1}}\right) \partial _{\alpha
}^{-4}\partial _{\alpha }H\overrightarrow{\theta }-\frac{AM^{4}}{8\pi ^{4}S}%
\partial _{\alpha }^{-4}P\overrightarrow{\theta }+\frac{\tau _{1}M^{2}}{4\pi
^{2}}\partial _{\alpha }^{-4}\partial _{\alpha }^{2}\overrightarrow{\theta }
\\
&&-\frac{\tau _{1}M^{2}}{4\pi ^{2}S}\left( \frac{\pi A\overline{\gamma }}{%
\tau _{1}M}-\frac{c}{\tau _{1}}\right) \partial _{\alpha }^{-4}\partial
_{\alpha }\overrightarrow{\gamma }.
\end{eqnarray*}

The mapping $\Gamma $ defined in (\ref{gamma_equation_new2}) is the composition of $\Theta $ and a mapping
identical to that which appears in \cite{AmbroseStraussWrightGlobBifurc}. \ It is shown
in \cite{AmbroseStraussWrightGlobBifurc} that $D\Gamma =\frac{cM}{\pi }HD\Theta ,$ so
by substituting our expression for $D\Theta $, we obtain%
\begin{eqnarray*}
D\Gamma &=&\frac{c\tau _{1}\overline{\gamma }M^{2}}{4\pi ^{2}S}\left( \frac{%
\overline{\gamma }}{\tau _{1}}-\frac{cAM}{\pi \tau _{1}}\right) H\partial
_{\alpha }^{-4}\partial _{\alpha }H\overrightarrow{\theta }-\frac{cAM^{5}}{%
8\pi ^{5}S}H\partial _{\alpha }^{-4}P\overrightarrow{\theta } \\
&&+\frac{c\tau
_{1}M^{3}}{4\pi ^{3}}H\partial _{\alpha }^{-4}\partial _{\alpha }^{2}%
\overrightarrow{\theta } -\frac{c\tau _{1}M^{3}}{4\pi ^{3}S}\left( \frac{\pi A\overline{\gamma }}{%
\tau _{1}M}-\frac{c}{\tau _{1}}\right) H\partial _{\alpha }^{-4}\partial
_{\alpha }\overrightarrow{\gamma }.
\end{eqnarray*}

Combining our results for $D\Theta $, $D\Gamma $, we write the linearization 
$L_{c}$ at $\left(
0,0;c\right) $ in matrix form:%
\begin{equation}
L_{c}%
\begin{bmatrix}
\overrightarrow{\theta } \\ 
\overrightarrow{\gamma }%
\end{bmatrix}%
:=%
\begin{bmatrix}
\overrightarrow{\theta }-D\Theta  \\ 
\overrightarrow{\gamma }-D\Gamma 
\end{bmatrix},
\label{linearized_equations}
\end{equation}
where
\begin{equation*}
L_{c}:=%
\begin{bmatrix}
L_{11} & L_{12} \\ 
L_{21} & L_{22}%
\end{bmatrix}%
\begin{bmatrix}
\overrightarrow{\theta } \\ 
\overrightarrow{\gamma }%
\end{bmatrix},%
\end{equation*}
with
\begin{eqnarray*}
L_{11} &:=&1-\frac{\tau _{1}\overline{\gamma }M}{4\pi S}\left( \frac{%
\overline{\gamma }}{\tau _{1}}-\frac{cAM}{\pi \tau _{1}}\right) \partial
_{\alpha }^{-4}\partial _{\alpha }H+\frac{AM^{4}}{8\pi ^{4}S}\partial
_{\alpha }^{-4}P-\frac{\tau _{1}M^{2}}{4\pi ^{2}}\partial _{\alpha
}^{-4}\partial _{\alpha }^{2}, \\
L_{12} &:=&\frac{\tau _{1}M^{2}}{4\pi ^{2}S}\left( \frac{\pi A\overline{%
\gamma }}{\tau _{1}M}-\frac{c}{\tau _{1}}\right) \partial _{\alpha
}^{-4}\partial _{\alpha }, \\
L_{21} &:=&-\frac{c\tau _{1}\overline{\gamma }M^{2}}{4\pi ^{2}S}\left( \frac{%
\overline{\gamma }}{\tau _{1}}-\frac{cAM}{\pi \tau _{1}}\right) H\partial
_{\alpha }^{-4}\partial _{\alpha }H+\frac{cAM^{5}}{8\pi ^{5}S}H\partial
_{\alpha }^{-4}P-\frac{c\tau _{1}M^{3}}{4\pi ^{3}}H\partial _{\alpha
}^{-4}\partial _{\alpha }^{2}, \\
L_{22} &:=&1+\frac{c\tau _{1}M^{2}}{4\pi ^{2}S}\left( \frac{A\overline{%
\gamma }}{\tau _{1}}-\frac{cM}{\pi \tau _{1}}\right) H\partial _{\alpha
}^{-4}\partial _{\alpha }.
\end{eqnarray*}

\subsubsection{Eigenvalue calculation}

The next step in applying Theorem \ref{general_bifurcation_theorem} to $%
\left( \theta -\Theta ,\gamma _{1}-\Gamma \right) $ is to find $c$ that
yield zero eigenvalues of $L_{c}$. \ To do so, we note the periodicity of $%
\left( \overrightarrow{\theta },\overrightarrow{\gamma }\right) $, and
examine the Fourier coefficients of $L_{c}$. \ Let $\mu$ be a general $2 \pi$-periodic map with convergent Fourier series. \ Noting our definition of $\widehat{\partial_{\alpha}^{-4} \mu }\left( k \right)$ in (\ref{eq:inverseDerivDef}), along with the
elementary results
\begin{eqnarray*}
\widehat{\partial _{\alpha }\mu }\left( k\right) &=&ik\widehat{\mu }\left(
k\right), \\
\widehat{\partial _{\alpha }^{2}\mu }\left( k\right) &=&-k^{2}\widehat{\mu }%
\left( k\right), \\
\widehat{H\mu }\left( k\right) &=&-i\func{ sgn}\left( k\right) \widehat{\mu }%
\left( k\right), \\
\widehat{P\mu }\left( k\right) &=&\left( 1-\delta _{0}\left( k\right)
\right) \widehat{\mu }\left( k\right),
\end{eqnarray*}%
we compute, for $k \neq 0$,
\begin{eqnarray*}
\widehat{L_{11}}\left( k\right) &=&1-\frac{\tau _{1}\overline{\gamma }M}{%
4\pi S}\left( \frac{\overline{\gamma }}{\tau _{1}}-\frac{cAM}{\pi \tau _{1}}%
\right) \frac{1}{\left\vert k\right\vert ^{3}}+\frac{AM^{4}}{8\pi ^{4}S}%
\frac{1}{k^{4}}+\frac{\tau _{1}M^{2}}{4\pi ^{2}}\frac{1}{k^{2}}, \\
\widehat{L_{12}}\left( k\right) &=&i\frac{\tau _{1}M^{2}}{4\pi ^{2}S}\left( 
\frac{\pi A\overline{\gamma }}{\tau _{1}M}-\frac{c}{\tau _{1}}\right) \frac{1%
}{k^{3}}, \\
\widehat{L_{21}}\left( k\right) &=&i\frac{c\tau _{1}\overline{\gamma }M^{2}}{%
4\pi ^{2}S}\left( \frac{\overline{\gamma }}{\tau _{1}}-\frac{cAM}{\pi \tau
_{1}}\right) \frac{1}{k^{3}}-i\frac{cAM^{5}}{8\pi ^{5}S}\frac{\func{sgn}\left(
k\right) }{k^{4}}-i\frac{c\tau _{1}M^{3}}{4\pi ^{3}}\frac{\func{sgn}\left( k\right) 
}{k^{2}}, \\
\widehat{L_{22}}\left( k\right) &=&1+\frac{c\tau _{1}M^{2}}{4\pi ^{2}S}%
\left( \frac{A\overline{\gamma }}{\tau _{1}}-\frac{cM}{\pi \tau _{1}}\right) 
\frac{1}{\left\vert k\right\vert ^{3}};
\end{eqnarray*}%
thus,
\begin{equation}\nonumber
\widehat{L_{c}%
\begin{bmatrix}
\overrightarrow{\theta } \\ 
\overrightarrow{\gamma }%
\end{bmatrix}%
}\left( k\right) 
=\widehat{L_{c}}\left( k\right) 
\begin{bmatrix}
\widehat{\overrightarrow{\theta }} \\ 
\widehat{\overrightarrow{\gamma }}%
\end{bmatrix}
=%
\begin{bmatrix}
\widehat{L_{11}}\left( k\right) & \widehat{L_{12}}\left( k\right) \\ 
\widehat{L_{21}}\left( k\right) & \widehat{L_{22}}\left( k\right)%
\end{bmatrix}%
\begin{bmatrix}
\overrightarrow{\theta } \\ 
\overrightarrow{\gamma }%
\end{bmatrix}%
.
\end{equation}

For $k=0$, $\widehat{P\mu }\left( k\right) =0$, so $\widehat{L_{c}}\left(
0\right) =I$. \ Clearly, when $k=0$, $1$ is an eigenvalue with multiplicity $%
2$. \ For $k\neq 0$, we compute the eigenvalues via Mathematica \cite{Mathematica}. \ In this
case, $1$ is also an eigenvalue, as is
\begin{equation*}
\lambda _{k}\left( c\right) :=1+\frac{M^{2}\tau _{1}}{4\pi ^{2}}\left\vert
k\right\vert ^{-2}+\frac{-c^{2}M^{3}+2Ac\overline{\gamma }M^{2}\pi -%
\overline{\gamma }^{2}M\pi ^{2}}{4\pi ^{3}S}\left\vert k\right\vert ^{-3}+%
\frac{AM^{4}}{8\pi ^{4}S}\left\vert k\right\vert ^{-4}.
\end{equation*}%
Note that $\lambda _{k}\left( c\right) $ is even with respect to $k$. \
Also, by basic Fourier series results, $\left\{ 1 \right\}\cup \left\{ \lambda _{k}\left( c\right) \right\} _{k}$ constitutes the
point spectrum of $L_{c}$. \ The eigenvector corresponding to $\lambda
_{k}\left( c\right) $ is%
\begin{equation}
v_{k}\left( c\right) :=%
\begin{bmatrix}
\frac{\func{sgn}\left( k\right) \text{ }i\pi }{cM} \\ 
1%
\end{bmatrix}%
,  \label{eigenvector_definition}
\end{equation}%
and thus
\begin{equation*}
\begin{bmatrix}
\frac{i\pi }{cM} \\ 
1%
\end{bmatrix}%
\exp \left( ik\alpha \right) \text{\ \ \ \ and \ \ }%
\begin{bmatrix}
-\frac{i\pi }{cM} \\ 
1%
\end{bmatrix}%
\exp \left( -ik\alpha \right) .
\end{equation*}%
are each eigenfunctions of $L_{c}$ \ However, we can take real and imaginary
parts of each, and obtain eigenfunctions
\begin{equation*}
\begin{bmatrix}
-\frac{\pi }{cM}\sin \left( k\alpha \right) \\ 
\cos \left( k\alpha \right)%
\end{bmatrix}%
\text{ \ \ \ \ and \ \ }%
\begin{bmatrix}
\frac{\pi }{cM}\cos \left( k\alpha \right) \\ 
\sin \left( k\alpha \right)%
\end{bmatrix}%
.
\end{equation*}%
Only the first is in $H_{\text{per},\text{odd}}^{2}\times H_{\text{per},0,\text{even}}^{1}$, thus, given
our chosen function space, we have that the dimension of the eigenspace of $%
\lambda _{k}\left( c\right) $ is one. \ We can then drop the absolute
values, and state, for $k>0$,
\begin{equation}
\lambda _{k}\left( c\right) =1+\frac{M^{2}\tau _{1}}{4\pi ^{2}}k^{-2}+\frac{%
-c^{2}M^{3}+2Ac\overline{\gamma }M^{2}\pi -\overline{\gamma }^{2}M\pi ^{2}}{%
4\pi ^{3}S}k^{-3}+\frac{AM^{4}}{8\pi ^{4}S}k^{-4}. \label{lambda_c_def}
\end{equation}

We summarize the spectral results thus far, along with with a few immediate
consequences, below:

\begin{proposition}
\label{spectral_proposition}Let $L_{c}$ be the linearization of $\left(\theta , \gamma _{1};c\right) \mapsto ( \theta -\Theta \left( \theta
,\gamma _{1};c\right) ,\gamma _{1}-\Gamma \left( \theta ,\gamma
_{1};c\right) ) $ at $\left( 0,0;c\right) $. \ The spectrum of $L_{c}$
is the set of eigenvalues $\left\{ 1\right\} \cup \left\{ \lambda _{k}\left(
c\right) :k\in 
\mathbb{N}
\right\} $, where $\lambda _{k}\left( c\right) $ is as defined in (\ref{lambda_c_def}). \ Each eigenvalue $\lambda $ of $L_{c}$ has
algebraic multiplicity equal to its geometric multiplicity, which we denote
\begin{equation*}
N_{\lambda }\left( c\right) :=\left\vert \left\{ k\in 
\mathbb{N}
:\lambda _{k}\left( c\right) =\lambda \right\} \right\vert ,
\end{equation*}%
and the corresponding eigenspace is%
\begin{equation*}
E_{\lambda }\left( c\right) :=\func{span}\left\{ 
\begin{bmatrix}
-\frac{\pi }{cM}\sin \left( k\alpha \right) \\ 
\cos \left( k\alpha \right)%
\end{bmatrix}%
:k\in 
\mathbb{N}
\text{ such that }\lambda _{k}\left( c\right) =\lambda \right\} .
\end{equation*}%
Also, for fixed $k$, if the inequality 
\begin{equation}
AM^{4}+\left( -2\overline{\gamma }^{2}M\pi ^{3}+2A^{2}\overline{\gamma }%
^{2}M\pi ^{3}\right) k+2M^{2}\pi ^{2}S\tau _{1}k^{2}+8\pi ^{4}Sk^{4}\geq 0
\label{nonstrict_ineq_condition}
\end{equation}%
holds, then the $c\in 
\mathbb{R}
$ for which $\lambda _{k}\left( c\right) =0$ is 
\begin{equation}
c_{\pm }\left( k\right) :=\frac{A\overline{\gamma }\pi }{M}\pm \sqrt{\frac{%
AM^{4}+\left( -2\overline{\gamma }^{2}M\pi ^{3}+2A^{2}\overline{\gamma }%
^{2}M\pi ^{3}\right) k+2M^{2}\pi ^{2}S\tau _{1}k^{2}+8\pi ^{4}Sk^{4}}{%
2kM^{3}\pi }},  \label{c_pm}
\end{equation}%
\newline
and this zero eigenvalue has multiplicity $N_{0}\left( c_{\pm }\left(
k\right) \right) \leq 2$. \ Specifically, if we define the polynomial (in $l$%
)%
\begin{equation*}
p\left( l;k\right) :=-AM^{4}+2kl\pi ^{2}S\left( 4\left(
k^{2}+kl+l^{2}\right) \pi ^{2}+M^{2}\tau _{1}\right) ,
\end{equation*}%
$p\left( \cdot ;k\right) $ has a single real root (denoted $l\left( k\right) 
$), and we have $N_{0}\left( c_{\pm }\left( k\right) \right) =2$ if and only
if $l\left( k\right) $ is a positive integer not equal to $k$.
\end{proposition}

\begin{proof}
The point spectrum of $L_{c}$, along with each $E_{\lambda }\left( c\right) $%
, was explicitly calculated above. \ The fact that the the geometric and
algebraic multiplicities are equal follows from the even/odd considerations
of the eigenfunctions. \ 

The result (\ref{c_pm}) follows from an easy computation, as $\lambda
_{k}\left( c\right) $ is quadratic in $c$. \ Next, we wish to make
statements about the multiplicity of the zero eigenvalue $\lambda _{k}\left(
c_{\pm }\left( k\right) \right) $. \ Using (\ref{lambda_c_def}) and (\ref{c_pm}), we compute%
\begin{equation}
\lambda _{l}\left( c_{\pm }\left( k\right) \right) =-\frac{\left( k-l\right)
\left( -AM^{4}+2kl\pi ^{2}S\left( 4\left( k^{2}+kl+l^{2}\right) \pi
^{2}+M^{2}\tau _{1}\right) \right) }{8kl^{4}\pi ^{4}S}. \label{lambdalk}
\end{equation}
Obviously, $\lambda _{l}\left( c_{\pm }\left( k\right) \right) =0$ when $l=k$%
. \ The other factor in the numerator of $\lambda _{l}\left( c_{\pm }\left(
k\right) \right) $ is precisely the polynomial $p$ defined above, which is
cubic in $l$. \ Using Mathematica, we compute its three roots, and label them $%
l_{1}\left( k\right) $, $l_{2}\left( k\right) $, $l_{3}\left( k\right) $: \ 
\begin{eqnarray*}
l_{1}\left( k\right) &:=&\frac{-B+4k^{8/3}\pi ^{4/3}S^{5/3}\left( C+3\sqrt{3}%
\sqrt{D}\right) ^{1/3}+\left( 2k^{2}S^{2}\left( C+3\sqrt{3}\sqrt{D}\right)
\right) ^{2/3}}{12\pi ^{4/3}\left( kS\right) ^{5/3}\left( C+3\sqrt{3}\sqrt{D}%
\right) ^{1/3}}, \\
l_{2}\left( k\right) &:=&\frac{zB-8k^{8/3}\pi ^{4/3}S^{5/3}\left( C+3\sqrt{3}%
\sqrt{D}\right) ^{1/3}-\overline{z}\left( 2k^{2}S^{2}\left( C+3\sqrt{3}\sqrt{%
D}\right) \right) ^{2/3}}{24\pi \left( kS\right) ^{5/3}\left( C+3\sqrt{3}%
\sqrt{D}\right) ^{1/3}}, \\
l_{3}\left( k\right) &:=&\overline{l_{2}\left( k\right) },
\end{eqnarray*}%
\newline
where
\begin{eqnarray*}
z &:=&1+i\sqrt{3}, \\
B &:=&2\cdot 2^{1/3}k^{2}S^{2}\left( 8\pi ^{8/3}k^{2}+3M^{2}\pi ^{2/3}\tau
_{1}\right), \\
C &:=&27AM^{4}+56k^{4}\pi ^{4}S+18k^{2}M^{2}\pi ^{2}S\tau _{1}, \\
D &:=&27A^{2}M^{2}+4k^{2}\pi ^{2}S^{2}\left( 3k^{2}\pi ^{2}+M^{2}\tau
_{1}\right) \left( 4k^{2}\pi ^{2}+M^{2}\tau _{1}\right) ^{2}\\
&&+4Ak^{2}M^{4}\pi
^{2}S\left( 28k^{2}\pi ^{2}+9M^{2}\tau _{1}\right).
\end{eqnarray*}%
We see $D>0$ given the nature of the constants in our problem; also, we have 
$B,C\in 
\mathbb{R}
$, so $l_{1}\left( k\right) $ is real. \ For $l_{2}\left( k\right) $ to be
real, we would need 
\begin{equation*}
\func{Im}\left[ zB-\overline{z}\left( 2k^{2}S^{2}\left( C+3\sqrt{3}\sqrt{D}%
\right) \right) ^{2/3}\right] =0.
\end{equation*}%
But,%
\begin{eqnarray*}
&&\func{Im}\left[ zB-\overline{z}\left( 2k^{2}S^{2}\left( C+3\sqrt{3}\sqrt{D}%
\right) \right) ^{2/3}\right] \\
&=&\sqrt{3}\left[ B+\left( 2k^{2}S^{2}\left( C+3\sqrt{3}\sqrt{D}\right)
\right) ^{2/3}\right]
\end{eqnarray*}
Thus, for $l_{2}\left( k\right) $ to be real, we would need%
\begin{equation*}
B=-\left( 2k^{2}S^{2}\left( C+3\sqrt{3}\sqrt{D}\right) \right) ^{2/3},
\end{equation*}
Since $S>0,k>0$, we see that $B$ is always positive, yet the right-hand-side
is always negative (recall that $D>0$). \ Thus, $l_{2}\left( k\right) $
(and, subsequently, $l_{3}\left( k\right) $ as well) necessarily has nonzero
imaginary part, and hence cannot be an integer. \ Therefore, when counting
the mulitplicity of zero eigenvalues of $L_{c}$, we only need to consider
the real root $l_{1}\left( k\right) $, which we label as $l\left( k\right) $%
. \ Given $k$ such that (\ref{nonstrict_ineq_condition}) holds, we have
that $N_{0}\left( c_{\pm }\left( k\right) \right) \leq 2$, since $p$ has one
real root $l\left( k\right) $. \ 

If $l\left( k\right) \neq k $ is a positive
integer, then we clearly have $N_{0}\left( c_{\pm }\left( k\right) \right)
=2 $ (since in this case both $k$ and $l\left( k\right) $ -- and only these
two positive integers -- correspond to the same zero eigenvalue).  Conversely, if $N_{0}\left( c_{\pm }\left( k\right) \right) = 2$, then the right-hand-side of (\ref{lambdalk}) must have a positive integer root $l \neq k$, and we established that such $l$ must be the real root $l(k)$ of the polynomial $p$.
\end{proof}

With this spectral information at hand, we are now ready to make some
statements about the crossing number of $L_{c}$.

\subsubsection{Necessary and sufficient conditions for odd crossing number}

\begin{proposition}
\label{odd_crossing_number_prop}Fix constants $A\in \left[ -1,1\right] ,%
\overline{\gamma }\in 
\mathbb{R}
$ and $S,\tau _{1},M>0$. \ Define the mapping $\left( \theta -\Theta ,\gamma
-\Gamma \right) $ as before, and let $L_{c}$ be its linearization at $\left(
0,0;c\right) $. \ Given fixed $k$, define $c_{\pm }\left( k\right) $ and $%
l\left( k\right) $ as in Proposition \ref{spectral_proposition}. \ Further,
put%
\begin{eqnarray}
&& \label{def_of_K} \\
K :=\{ k\in 
\mathbb{N}
&:&AM^{4}+\left( -2\overline{\gamma }^{2}M\pi ^{3}+2A^{2}\overline{\gamma }^{2}M\pi ^{3}\right)
k+2M^{2}\pi ^{2}S\tau _{1}k^{2}+8\pi ^{4}Sk^{4}>0 \notag \\ 
&&\text{ and }l\left(
k\right) \notin 
\mathbb{N}
\setminus \left\{ k\right\} \} \notag.  
\end{eqnarray}%
Then, $L_{c}$ has an odd crossing number (specifically, the crossing number
is one) at $c=c_{\pm }\left( k\right) $ (which is real) if and only if $k\in
K$. \ Furthermore, $\left\vert K\right\vert =\infty $. \ 
\end{proposition}

\begin{proof}
First, assume $k\in K$. \ The first condition in the definition of $K$
ensures (\ref{nonstrict_ineq_condition}) holds in Proposition \ref%
{spectral_proposition}, so we have that $c=c_{\pm }\left( k\right) $ is
real, and yields a zero eigenvalue of $L_{c}$, so $N_{0}\left( c_{\pm
}\left( k\right) \right) \geq 1$. \ The second condition ensures that $%
N_{0}\left( c_{\pm }\left( k\right) \right) <2$ by the last conclusion of
Proposition \ref{spectral_proposition}. \ Thus, $N_{0}\left( c_{\pm }\left(
k\right) \right) =1$. \ 

For a direct calculation of the crossing number, we examine a perturbation of this zero eigenvalue, computed via Mathematica:
\begin{eqnarray}
&&\lambda _{k}\left( c_{\pm }\left( k\right) +\varepsilon \right) \label{perturbedeigk}\\ 
&=&\mp \frac{M^{3}}{2k^{3}\pi ^{3}S}\sqrt{\frac{AM^{4}+\left( -2\overline{\gamma }^{2}M\pi
^{3}+2A^{2}\overline{\gamma }^{2}M\pi ^{3}\right) k+2M^{2}\pi ^{2}S\tau
_{1}k^{2}+8\pi ^{4}Sk^{4}}{2kM^{3}\pi }}\varepsilon \notag \\
&&-\frac{M^{3}}{4k^{3}\pi
^{3}S}\varepsilon ^{2}.  \notag
\end{eqnarray}%
Note that this expression is exact, since from (\ref{lambda_c_def}) we have that $\lambda_k \left(c\right)$ is quadratic in $c$. \ Also, in the leading-order term of (\ref{perturbedeigk}), the expression under the radical is
identical to that which is under the radical of (\ref{c_pm}); yet, we have a strict inequality in the
first condition in the definition of $K$. \ Thus, since this leading-order
coefficient is positive, we see that $\lambda _{k}\left( c_{\pm }\left(
k\right) +\varepsilon \right) $ changes sign as $\varepsilon $ passes over
zero. \ Since we have a multiplicity of one, we see that there is an odd crossing
number at $c=c_{\pm }\left( k\right) $ from a direct application of
Definition \ref{odd_crossing_number_def} (and this crossing number is in
fact one, as exactly one eigenvalue, counted with multiplicity, is changing
sign as $\varepsilon $ passes over zero. \ 

Now, assume an odd crossing number at a real $c=c_{\pm }\left( k\right) $
for some $k\in 
\mathbb{N}
$. \ By definition of $c_{\pm }\left( k\right) $, we have $N_{0}\left(
c_{\pm }\left( k\right) \right) \geq 1$, and by Proposition \ref%
{spectral_proposition}, we have $N_{0}\left( c_{\pm }\left( k\right) \right)
\leq 2$. \ Thus, either $N_{0}\left( c_{\pm }\left( k\right) \right) =1$ or $%
N_{0}\left( c_{\pm }\left( k\right) \right) =2$. \ We show that the former
implies $k\in K$, and that the latter leads to a contradiction. \ In the
case $N_{0}\left( c_{\pm }\left( k\right) \right) =1$, we necessarily have $%
l\left( k\right) \notin 
\mathbb{N}
\setminus \left\{ k\right\} $; otherwise, the multiplicity would be $2$. \
Morover, $c_{\pm }\left( k\right) $ must be real, and since the crossing
number is odd, $\lambda _{k}\left( c_{\pm }\left( k\right) +\varepsilon
\right) $ (which is the only eigenvalue perturbing from the zero eigenvalue,
assumed in this case to be of multiplicity 1) must change sign as $%
\varepsilon $ passes over $0$. Thus, by the same calculation (\ref{perturbedeigk}), we need the leading-order term to be nonzero,
which forces the strict inequality in the first condition in the definition
of $K$. \ Thus, we have $k\in K$.

If $N_{0}\left( c_{\pm }\left( k\right) \right) =2$, then clearly $l\left(
k\right) \neq k$ is a positive integer. \ A Mathematica calculation shows%
\begin{eqnarray}
&&\;\;\; \lambda _{l\left( k\right) }\left( c_{\pm }\left( k\right) +\varepsilon
\right) \label{perturb_eig_lk}\\
&=&-\frac{\left(k - l(k)\right)\left(-AM^4+2k\,l(k) \pi^2S\left(4\left(k^2 + k\, l(k) + [l(k)]^2\right)\pi^2 + M^2 \tau_1\right)\right)}{8k\, [l(k)]^4 \pi^4 S} \notag \\
&& \mp \frac{M^3}{2 [l(k)]^3\pi^3S}\sqrt{\frac{AM^{4}+\left( -2\overline{\gamma }^{2}M\pi
^{3}+2A^{2}\overline{\gamma }^{2}M\pi ^{3}\right) k+2M^{2}\pi ^{2}S\tau
_{1}k^{2}+8\pi ^{4}Sk^{4}}{2kM^{3}\pi }} \varepsilon \notag \\
&&-\frac{M^3}{4 [l(k)]^3 \pi^3 S} \varepsilon^2. \notag
\end{eqnarray}%
As expected, the zeroth-order (in $\varepsilon$) term of (\ref{perturb_eig_lk}) is precisely $\lambda_{l(k)}\left(c_\pm(k)\right)$ (see (\ref{lambdalk})), which vanishes by definition of $l(k)$. \ Then, we can see that the first-order term of (\ref{perturb_eig_lk}) is real and nonzero if and only if the first-order term of (\ref{perturbedeigk}) is real and nonzero. \ Thus, as $\varepsilon $
passes over zero, either (i) both $\lambda _{k}\left( c_{\pm }\left(
k\right) +\varepsilon \right) $ and $\lambda _{l\left( k\right) }\left(
c_{\pm }\left( k\right) +\varepsilon \right) $ change signs, or (ii) neither
of them change signs. \ In either case, we cannot have an odd crossing
number, which contradicts our assumption.

Finally, we need to show that $\left\vert K\right\vert =\infty $%
. \ First, we see that $8\pi ^{4}S$ is positive, so for sufficiently large $k$,
the first condition in the definition of $K$ is satisfied. \ For the second
condition, we need to check the behavior of $l\left( k\right) $. \ Perform
the subsitution $k=1/\delta $ ($\delta >0$), and examine the Taylor series
expansion of $l\left( 1/\delta \right) $ about $\delta =0$ (computed via
Mathematica):%
\begin{equation*}
l\left( \frac{1}{\delta }\right) =\frac{AM^{4}}{8\pi ^{4}S} \delta ^{3}
+O\left( \delta ^{5}\right) .
\end{equation*}%
We see that%
\begin{equation*}
\lim_{\delta \rightarrow 0}l\left( \frac{1}{\delta }\right) =0,
\end{equation*}%
\newline
so%
\begin{equation*}
\lim_{k\rightarrow \infty }l\left( k\right) =0
\end{equation*}%
as well. \ Thus, for sufficiently large $k$, the second condition in the
definition of $K$ holds. \ Since both the first and second conditions hold
for sufficiently large $k$, we have that $\left\vert K\right\vert =\infty $.
\ 
\end{proof}

\begin{remark}
There are other, ``indirect" methods for providing sufficient conditions for
an odd crossing number. \ One result, which appears in \cite{KielhoferBifurcationBook},
states that if $0$ is a geometrically simple eigenvalue of $D_{x}F\left(
0,c_{0}\right) $, and%
\begin{equation}
\left( Q\left[ \frac{d}{dc}D_{x}F\left( 0,c\right) \right] _{c=c_{0}}\right)
v\neq 0,  \label{indirect_condition}
\end{equation}%
\newline
where $v$ is an arbitrary element of the null space of $D_{x}F\left(
0,c_{0}\right) $, and $Q$ is the projection onto the cokernel of $%
D_{x}F\left( 0,c_{0}\right) $ (note the dimensions in this case are such
that the left-hand side is a scalar quantity), then $D_{x}F\left( 0,c\right) 
$ has an odd crossing number at $c=c_{0}$. \ For our problem, we note that
for matrices $A$, coker$(A)=\ker \left( A^{\ast }\right) $, and thus we
examine 
\begin{equation}
w_{k}\left( c_{0}\right) \cdot \left( \left[ \frac{d}{dc}\widehat{L_{c}}\left(
k\right) \right] _{c=c_{0}}v_{k}\left( c_{0}\right) \right)
\label{indirect_method_condition}
\end{equation}%
\newline
where $v_{k}\left( c_{0}\right) $ is as defined in (\ref%
{eigenvector_definition}), $w_{k}\left( c_{0}\right) $ is the eigenvector
corresponding to the zero eigenvalue of $\left[ \widehat{L_{c_{0}}}\left(
k\right) \right] ^{\ast },$ and the dot indicates the usual Euclidean dot
product. \ 

We arrive at precisely the same condition for (\ref%
{indirect_method_condition}) to be nonzero as the strict inequality in the
the definition of $K$ (\ref{def_of_K}). \ This, coupled with the assumption
that the null space is one-dimensional, allows us to conclude that if $k\in
K $, then the crossing number is odd. \ Unfortunately, (\ref%
{indirect_condition}) is stated as only a sufficient condition for odd
crossing number, and so we could not conclude the ``only if" direction of
Proposition \ref{odd_crossing_number_prop} by using this method alone. \
Because of this, and because the nature of the problem fortunately allowed
for the method to be manageable and conclusive, we favored a ``direct"
verification of an odd crossing number (namely calculations (\ref%
{perturbedeigk}) and (\ref{perturb_eig_lk})). \ 
\end{remark}

\subsubsection{Global bifurcation conclusion}

We can now apply Theorem \ref{general_bifurcation_theorem}, and cast the
conclusions of this abstract theorem in the language of our problem.

\begin{theorem}
\label{global_bifurc_theorem_for_problem}Define $U_{0,0}:=\cup
_{b,h>0}U_{b,h} $, where $U_{b,h}$ is as defined in (\ref{Ubh_definition}). 
Let $S\subseteq U_{0,0}$ be the closure (in $%
H_{\text{per},\text{odd}}^{2}\times H_{\text{per},0,\text{even}}^{1}\times 
\mathbb{R}
$) of the set of nontrivial (i.e. $\theta, \gamma _{1}$ are not both zero)
solutions of the traveling wave problem $\left( \theta -\Theta \left( \theta
,\gamma _{1};c\right) ,\gamma _{1}-\Gamma \left( \theta ,\gamma
_{1};c\right) \right) = 0 $. \ Furthermore, let $c_{\pm }\left(
k\right) $ and $K$ be as in Propositions \ref{spectral_proposition} and \ref{odd_crossing_number_prop}. \ For fixed $k\in K$, define $C_{\pm }\left(
k\right) $ to be the connected component of $S$ that contains $\left(
0,0;c_{\pm }\left( k\right) \right) $. \ Then\textit{, one of the following
alternatives is valid:}\\

\begin{enumerate}[label={\upshape(\Roman*):}]
\item $C_{\pm }\left( k\right) $ is unbounded; or

\item $C_{\pm }\left( k\right) =C_{+}\left( j\right) $ or $C_{\pm
}\left( k\right) =C_{-}\left( j\right) $ for some $j\in K$ with $%
j\neq k$; or 

\item $C_{\pm }\left( k\right) $ contains a point on the
boundary of $U_{0,0}$.\\
\end{enumerate}
\end{theorem}

\begin{proof}
Assume $k\in K$. \ Given $b,h>0$, we have by Proposition \ref{mapping_proposition} that $%
\left( \Theta ,\Gamma \right) $ is compact on $U_{b,h}$, and by Proposition %
\ref{odd_crossing_number_prop}, there is an odd crossing number of the
linearization at $c=c_{\pm }\left( k\right) $. \ Since the conditions for
Theorem \ref{general_bifurcation_theorem} are met, we can conclude that one
of outcomes (i), (ii), (iii) in the conclusion of this theorem occur (using $%
U=U_{b,h}$). \ The outcomes (i) (ii) correspond exactly with outcomes (I),
(II) above. \ Taking a union over all $b,h>0$, outcome (iii) yields outcome
(III) above via a simple topological argument.
\end{proof}

\subsection{Proof of main theorem}

We are now ready to prove Theorem \ref{main_theorem}. \ With the bifurcation
results of Theorem \ref{global_bifurc_theorem_for_problem} at hand, this
proof will largely be comprised of matching the outcomes (I) -- (III) above
with the outcomes (a) -- (e) in Theorem \ref{main_theorem}. \ Many of these
conclusions can be reached through arguments identical to (or closely analogous
to) those in \cite{AmbroseStraussWrightGlobBifurc}.

\begin{proof}
First, note that Proposition \ref{odd_crossing_number_prop} gives us that $%
\left\vert K\right\vert =\infty $; hence, by Theorem \ref%
{global_bifurc_theorem_for_problem}, we have a countable number of connected
sets of the form $C_{\pm }\left( k\right) $ ($k\in K$) that satisfy one of
outcomes (I) -- (III). \ 

We can immediately see that Outcome (II) means (e).

Next, consider Outcome (III). \ Recall the definition of $U_{0,0}$ in
Theorem \ref{global_bifurc_theorem_for_problem}, and note that we can write%
\begin{eqnarray*}
U_{0,0}=\left\{ \left( \theta ,\gamma _{1};c\right) \in X
: \overline{\cos \theta }>0,\widetilde{Z}\left[ \theta \right] \in C_{0}^{2} \text{ and }\widetilde{Z}\left[ \Theta \left( \theta ,\gamma _{1};c\right) %
\right] \in C_{0}^{5}\right\},
\end{eqnarray*}%
where
\begin{equation*}
X = H_{\text{per},\text{odd}}^{2}\times H_{\text{per},0,\text{even}}^{1}\times 
\mathbb{R}.
\end{equation*}
Outcome (III) means either $\overline{\cos \theta }=0$ or $\widetilde{Z}%
\left[ \theta \right] =\widetilde{Z}\left[ \Theta \left( \theta ,\gamma
_{1};c\right) \right] \notin C_{0}^{5}$. \ As in \cite{AmbroseStraussWrightGlobBifurc}, 
$\overline{\cos \theta }=0$ implies (a). \ If $\widetilde{Z}\left[ \theta %
\right] \notin C_{0}^{5}$, we have that the interface self-intersects by the
same argument as in \cite{AmbroseStraussWrightGlobBifurc} (the $s=5$ regularity
does not affect this argument). \ This is outcome (d). \ 

Outcome (I) (i.e. unboundedness of the solution set) can lead to more
outcomes in this main theorem. \ If $C_{\pm }\left( k\right) $ is unbounded,
then it contains a sequence of solutions $\left( \theta _{n},\gamma
_{1,n};c_{n}\right) $ in $U$ such that%
\begin{equation}
\lim_{n\rightarrow \infty }\left( \left\vert c_{n}\right\vert +\left\Vert
\theta _{n}\right\Vert _{H_{\text{per}}^{2}}+\left\Vert \gamma _{1,n}\right\Vert
_{H_{\text{per}}^{1}}\right) =\infty. \label{diverging_norms}
\end{equation}
We first note that, as in \cite{AmbroseStraussWrightGlobBifurc}, if (a) does not hold,
then $\sigma _{n}$ is bounded above independently of $n$. \ For the
remainder of this proof, assume (a) does not hold, and hence $\sigma _{n}$
is bounded above independently of $n$. \ At least one of the three terms of (\ref{diverging_norms}) must diverge; we subsequently examine each case.

If $\left\vert c_{n}\right\vert
\rightarrow \infty $, yet $\left\Vert \theta _{n}\right\Vert
_{H_{\text{per}}^{2}}+\left\Vert \gamma _{1,n}\right\Vert _{H_{\text{per}}^{1}}$ is
bounded, then either $\overline{\gamma }\neq 0$ or $A\neq 0$ are violated as
in \cite{AmbroseStraussWrightGlobBifurc} (yet $\left\vert c_{n}\right\vert \rightarrow \infty $ is
outcome (f), which may possible, independent of the other outcomes, when $%
\overline{\gamma }=0$ and $A=0$). \ To see this, first recall that (\ref{normal_equation_new2})
implies%
\begin{equation*}
\left\Vert c_{n}\sin \left( \theta _{n}\right) \right\Vert
_{H^{2}}=\left\Vert \func{Re}\left( W_{n}^{\ast }N_{n}\right) \right\Vert
_{H^{2}}.
\end{equation*}%
Lemma 5 of \cite{AmbroseStraussWrightGlobBifurc} shows us that the right-hand side is
bounded, so we have that $c_{n}\sin \left( \theta _{n}\right) $ is bounded
in $H^{2}$. \ Since by assumption $\left\vert c_{n}\right\vert \rightarrow
\infty ,$ this forces $\left\Vert \sin \left( \theta _{n}\right) \right\Vert
_{H_{\text{per}}^{2}}$ to approach $0$; then, by Sobolev embedding, we have $\sin
\left( \theta _{n}\right) \rightarrow 0$ uniformly. \ Thus, $\theta _{n}$
must converge to a multiple of $\pi $; but, by continuity and the fact that $%
\theta $ is an odd function, we have $\theta _{n}\rightarrow 0$ uniformly,
and clearly $\left\vert \cos \left( \theta _{n}\right) \right\vert
\rightarrow 1$ (uniformly) as well. \ 

Now, recall (\ref{tan_equation2}):
\begin{eqnarray*}
0 &=&-\frac{S}{\tau _{1}\sigma ^{2}}\left( \partial _{\alpha }^{4}\theta +%
\frac{3\theta _{\alpha }^{2}\theta _{\alpha \alpha }}{2}-\tau _{1}\sigma
^{2}\theta _{\alpha \alpha }\right) -\frac{2\widetilde{A}\sigma }{\tau _{1}}%
\left( \cos \theta \right) _{\alpha } \\
&&+\frac{1}{\tau _{1}}\left( \left( c\cos \theta -\func{Re}\left( W^{\ast
}T\right) \right) \gamma \right) _{\alpha } \\
&&-\frac{A}{\tau _{1}}\left( \frac{%
\partial _{\alpha }\left( \gamma ^{2}\right) }{4\sigma }+2\sigma \sin \theta
+\sigma \partial _{\alpha }\left\{ \left( c\cos \theta -\func{Re}\left(
W^{\ast }T\right) \right) ^{2}\right\} \right),
\end{eqnarray*}%
or%
\begin{eqnarray*}
0 &=&S\left( -\frac{\partial _{\alpha }^{4}\theta }{\sigma ^{2}}-\frac{%
3\theta _{\alpha }^{2}\theta _{\alpha \alpha }}{2\sigma ^{2}}+\tau
_{1}\theta _{\alpha \alpha }\right) -2\widetilde{A}\sigma \left( \cos \theta
\right) _{\alpha } \\
&&+\left( \left( c\cos \theta -\func{Re}\left( W^{\ast }T\right) \right)
\gamma \right) _{\alpha } \\
&&-2A\left( \frac{\partial _{\alpha }\left( \gamma
^{2}\right) }{\sigma }+\sigma ^{2}\sin \theta +\frac{1}{2}\sigma \partial
_{\alpha }\left\{ \left( c\cos \theta -\func{Re}\left( W^{\ast }T\right)
\right) ^{2}\right\} \right) .
\end{eqnarray*}%
We integrate twice with respect to $\alpha $, and substitute our sequences of
solutions:%
\begin{eqnarray*}
0 &=&S\left( -\frac{\partial _{\alpha }^{2}\theta _{n}}{\sigma _{n}^{2}}%
-\int\limits^{\alpha }\frac{\left( \partial _{\alpha }\theta _{n}\right)
^{3}}{2\sigma ^{2}}\text{ }d\alpha +\tau _{1}\theta _{n}\right) -2\widetilde{%
A}\sigma _{n}\int\limits^{\alpha }\cos \theta _{n}d\alpha \\
&&+\int\limits^{\alpha } \left( c_{n}\cos \theta _{n}-\func{Re}%
\left( W^{\ast }T\right) \right) \gamma _{n} d\alpha \\
&&-2A\int\limits^{\alpha } \left\{\left[ \frac{1}{8}\frac{%
\gamma _{n}^{2}}{\sigma _{n}}+\sigma _{n}^{2}\int\limits^{\alpha }\sin
\theta _{n}\text{ }d\alpha +\frac{\sigma _{n}}{2}\left( c_{n}\cos \theta
_{n}-\func{Re}\left( W_{n}^{\ast }T_{n}\right) \right) ^{2}\right] \right\}
d\alpha .
\end{eqnarray*}
Given that $\left\Vert \theta _{n}\right\Vert _{H_{\text{per}}^{2}}$ is bounded, we
have that%
\begin{equation*}
S\left( -\frac{\partial _{\alpha }^{2}\theta _{n}}{\sigma _{n}^{2}}%
-\int\limits^{\alpha }\frac{\left( \partial _{\alpha }\theta _{n}\right)
^{3}}{2\sigma ^{2}}\text{ }d\alpha +\tau _{1}\theta _{n}\right)
\end{equation*}%
is bounded in $L_{\text{per}}^{2}$, as is the ``mass" term%
\begin{equation*}
2\widetilde{A}\sigma _{n}\int\limits^{\alpha }\cos \theta _{n} \,d\alpha .
\end{equation*}%
Thus, 
\begin{eqnarray}
&& \int\limits^{\alpha } \left( c_{n}\cos \theta _{n}-\func{Re}\left(
W^{\ast }T\right) \right) \gamma _{n} d\alpha \label{bounded_in_L2_for_outcome_f_results2} \\
&&-2A \int\limits^{\alpha } \left\{ \left[ \frac{1}{8}\frac{\gamma
_{n}^{2}}{\sigma _{n}}+\sigma _{n}^{2}\int\limits^{\alpha }\sin \theta _{n}%
\text{ }d\alpha +\frac{\sigma _{n}}{2}\left( c_{n}\cos \theta _{n}-\func{Re}%
\left( W_{n}^{\ast }T_{n}\right) \right) ^{2}\right] \right\} d\alpha \notag
\end{eqnarray}%
must be bounded in $L_{\text{per}}^{2}$. \ 

Now, assume that $A=0$. \ With this
assumption, we are left with the conclusion that%
\begin{equation*}
\int\limits^{\alpha }\left( c_{n}\cos \theta _{n}-\func{Re}\left( W^{\ast
}T\right) \right) \gamma _{n} \, d\alpha
\end{equation*}
is bounded in $L_{\text{per}}^{2}$. \ Again, Lemma 5 of \cite{AmbroseStraussWrightGlobBifurc}
shows $\func{Re}\left( W^{\ast }T\right) $ is bounded in $H_{\text{per}}^{1}$, so
since $\gamma _{n}$ is also assumed to be bounded in $H_{\text{per}}^{1}$, we have
that%
\begin{equation*}
\int\limits^{\alpha }\func{Re}\left( W^{\ast }T\right) \gamma _{n}\,d\alpha
\end{equation*}%
is bounded in $H_{\text{per}}^{2}$, and hence is also bounded in $L_{\text{per}}^{2}$. \ Thus,%
\begin{equation*}
\int\limits^{\alpha }c_{n}\cos \left( \theta _{n}\right) \gamma _{n} \,d\alpha
\end{equation*}
is subsequently bounded in $L_{\text{per}}^{2}$ as well. \ Since $c_{n}\rightarrow \infty \,$, this forces%
\begin{equation*}
\int\limits^{\alpha }\gamma _{n} \,d\alpha \rightarrow 0 \;\left(\text{in } L_{\text{per}}^{2} \right).
\end{equation*}
However, we can write $\gamma _{n}=\gamma _{1,n}+\overline{%
\gamma }$, where $\gamma _{1,n}$ has mean zero. \ Since 
\begin{equation*}
\int\limits^{%
\alpha }\gamma _{1,n}d\alpha \rightarrow 0 \;\left(\text{in } L_{\text{per}}^{2} \right),
\end{equation*}
we need
\begin{equation*}
\int\limits^{\alpha }\overline{\gamma }d\alpha = \overline{\gamma }\alpha \rightarrow 0 \;\left(\text{in } L_{\text{per}}^{2} \right).
\end{equation*}
This forces $\overline{\gamma }=0$.

Now, suppose $A\neq 0$. \ Divide (\ref{bounded_in_L2_for_outcome_f_results2}%
) by $c_{n}^{2}$:%
\begin{eqnarray}
&&\int\limits^{\alpha } \left( \frac{1}{c_{n}}\cos \theta _{n}-\frac{1%
}{c_{n}^{2}}\func{Re}\left( W^{\ast }T\right) \right) \gamma _{n}d\alpha \label{divided_by_cn2_2} \\
&&-2A \int\limits^{\alpha } \left\{ \left[ \frac{1}{8}\frac{\gamma _{n}^{2}}{c_{n}^{2}\sigma _{n}}+\frac{\sigma _{n}^{2}%
}{c_{n}^{2}}\int\limits^{\alpha }\sin \theta _{n}\text{ }d\alpha +\frac{%
\sigma _{n}}{2c_{n}^{2}}\left( c_{n}\cos \theta _{n}-\func{Re}\left(
W_{n}^{\ast }T_{n}\right) \right) ^{2}\right] \right\} d\alpha . \notag
\end{eqnarray}%
Since (\ref{bounded_in_L2_for_outcome_f_results2}) is bounded in $L_{\text{per}}^{2}$, we
have that (\ref{divided_by_cn2_2}) must approach $0$. \ Examining each term
of (\ref{divided_by_cn2_2}), we see that all but possibly%
\begin{equation*}
\int\limits^{\alpha }\frac{\sigma _{n}}{2c_{n}^{2}}\left( c_{n}\cos \theta
_{n}-\func{Re}\left( W_{n}^{\ast }T_{n}\right) \right) ^{2}d\alpha
\end{equation*}%
clearly approach $0$. \ After expanding the integrand, we see that all terms
would in fact approach zero on their own merit except for the leading-order
(in $c_{n}$) term%
\begin{equation*}
\int\limits^{\alpha }\frac{\sigma _{n}}{2c_{n}^{2}}c_{n}^{2}\cos ^{2}\theta
_{n}d\alpha =\frac{1}{2}\int\limits^{\alpha }\sigma _{n}\cos ^{2}\theta
_{n}d\alpha .
\end{equation*}%
Since all other terms of (\ref{divided_by_cn2_2}) approaches zero, this forces 
\begin{equation*}
\frac{1}{2}\int\limits^{\alpha }\sigma _{n}\cos ^{2}\theta _{n}d\alpha
\rightarrow 0\text{.}
\end{equation*}%
Thus,%
\begin{equation*}
\sigma _{n}\cos ^{2}\theta _{n}\rightarrow 0\text{,}
\end{equation*}%
\newline
which is a contradiction, since $\left\vert \cos \theta _{n}\right\vert
\rightarrow 1$ from earlier. \ 

In whole, we have a contradiction between the statements (i) $\left\vert
c_{n}\right\vert \rightarrow \infty $ but $\left\Vert \theta _{n}\right\Vert
_{H_{\text{per}}^{2}}+\left\Vert \gamma _{1,n}\right\Vert _{H_{\text{per}}^{1}}$ is
bounded,\ and (ii) either $\overline{\gamma }\neq 0$ or $A\neq 0$. \ Thus,
if we assume either $\overline{\gamma }\neq 0$ or $A\neq 0$ and $\left\vert
c_{n}\right\vert \rightarrow \infty $, we must necessarily have $\left\Vert
\theta _{n}\right\Vert _{H_{\text{per}}^{2}}+\left\Vert \gamma _{1,n}\right\Vert
_{H_{\text{per}}^{1}}$ additionally unbounded; hence, in this case, outcome (f)
implies other outcomes as in the \cite{AmbroseStraussWrightGlobBifurc}. \ But, if both $%
A=0$ and $\overline{\gamma }=0$, we do not exclude the possibility $\left\vert
c_{n}\right\vert \rightarrow \infty $ but $\left\Vert \theta _{n}\right\Vert
_{H_{\text{per}}^{2}}+\left\Vert \gamma _{1,n}\right\Vert _{H_{\text{per}}^{1}}$ is bounded, so we
list (f) as another possible outcome.

With the case $\left\vert c_{n}\right\vert \rightarrow \infty $ handled, we now turn our attention to the case in which $\left\Vert \theta
_{n}\right\Vert _{H_{\text{per}}^{2}}$ (the second term of (\ref{diverging_norms})) diverges. \ This means that one of $\theta
_{n}$, $\partial _{\alpha }\theta _{n}$, or $\partial _{\alpha }^{2}\theta
_{n}$ diverge. \ Recall that $\kappa \left( \alpha \right) =\partial
_{\alpha }\theta \left( \alpha \right) /\sigma $, so $\partial _{\alpha
}\kappa \left( \alpha \right) =\partial _{\alpha }^{2}\theta \left( \alpha
\right) /\sigma $. \ Since 
\begin{equation*}
2\pi \overline{\cos \theta }=\int\nolimits_{0}^{2\pi }\cos \left( \theta
\left( \alpha ^{\prime }\right) \right) d\alpha ^{\prime }=\frac{M}{\sigma },
\end{equation*}%
we cannot have $\sigma _{n}\rightarrow 0$. \ Since $\sigma _{n}$ is bounded
above as well, then if $\partial _{\alpha }^{2}\theta _{n}$ diverges, then
so does the derivative of curvature. \ This is outcome (b). \ 

Finally, it is shown in \cite{AmbroseStraussWrightGlobBifurc} that $\left\Vert \gamma _{1,n}\right\Vert
_{H_{\text{per}}^{1}}\rightarrow \infty $ implies either outcome (a) or outcome (c). \ 

Note that if curvature or
the jump of the tangential component of fluid velocity themselves are
arbitrarily large, then (respectively) (b) or (c) occur as well, so we omit
these as separate outcomes. \ 
\end{proof}

In the next section, we proceed to numerically compute some of these diverse
solution curves.

\section{Numerical methods and results}

\subsection{Methods}

For our numerical computations, we employ methods very similar to those in \cite{AkersAmbroseWrightTravelingWavesSurfaceTension}.  Our computations use the version of our equations
given by (\ref{tan_equation_for_numerical}) and (\ref{normal_equation_new2}%
). \ We specify the horizontal domain width $M=2\pi $ ($\alpha \in \left[
-\pi ,\pi \right] $), and project each of $\theta ,\gamma $ onto a
finite-dimensional Fourier space:%
\begin{equation*}
\theta \left( \alpha \right) =\sum\limits_{k=-N}^{k=N}a_{k}\exp \left(
ik\alpha \right) ,\text{ \ \ \ \ \ }\gamma \left( \alpha \right)
=\sum\limits_{k=-N}^{k=N}b_{k}\exp \left( ik\alpha \right) .
\end{equation*}%
As in the formulation of our problem, we work with odd, real $\theta $ and
even, real $\gamma $. \ This forces $a_{-k}=-a_{k}$ (so $a_{0}=0$) and $%
b_{-k}=b_{k}$ (and clearly $b_{0}=\overline{\gamma }$). \ Thus, with $%
\overline{\gamma }$ specified \textit{a priori}, we can see that a
traveling wave solution $\left( \theta ,\gamma ;c\right) $ is determined by
the $2N$ coefficients $a_{1},\dots ,a_{N},b_{1},\dots ,b_{N}$, along with
the wave speed $c$. \ To solve for these $2N+1$ values, we project both
sides of equations (\ref{tan_equation_for_numerical}), (\ref%
{normal_equation_new2}) onto each basis element $\exp \left( ik\alpha
\right) $, $1\leq k\leq N$. \ This yields a system of $2N$ algebraic
equations; to complete the system, we include another equation that allows
us to specify the amplitude of the solution.

The most difficult, non-obvious portion of the computation of these
algebraic equations is, perhaps, the computation of the Birkhoff-Rott
integral $W^{\ast }$ in Fourier space. \ However, as in \cite{AmbroseStraussWrightGlobBifurc} (and what was used in our ``identity plus compact"
formulation (\ref{theta_equation_new}), (\ref{gamma_equation_new})), we can write $W^{\ast }$ as the sum%
\begin{equation*}
W^{\ast }=\frac{1}{2}H\left( \frac{\gamma }{z_{\alpha }}\right) +\mathcal{K} \left[ z%
\right] \gamma \text{,}
\end{equation*}%
where, as before, $H$ is the Hilbert transform, and the remainder $\mathcal{K}\left[ z%
\right] \gamma $ can be explicitly written as%
\begin{eqnarray*}
&&\mathcal{K}\left[ z\right] \gamma \left( \alpha \right) =  \\
&&\frac{1}{4\pi i}\func{PV}\int\nolimits_{0}^{2\pi }\gamma \left( \alpha
^{\prime }\right) \left[ \cot \left( \frac{1}{2}\left( z\left( \alpha
\right) -z\left( a^{\prime }\right) \right) \right) -\frac{1}{\partial
_{\alpha ^{\prime }}z\left( \alpha ^{\prime }\right) }\cot \left( \frac{1}{2}%
\left( \alpha -\alpha ^{\prime }\right) \right) \right] d\alpha ^{\prime }.
\end{eqnarray*}%
The Hilbert transform $H$ is easily computed in Fourier space, as
\begin{equation*}
\widehat{H\mu }\left( k\right) =-i \func{sgn} \left( k\right) \widehat{\mu }\left(
k\right).
\end{equation*} 
The remainder $\mathcal{K} \left[ z\right] \gamma $ is computed with the
trapezoid rule, in an ``alternating" sense (i.e. to evaluate this integral at
an ``even" grid point, we sum the ``odd" nodes, and vice-versa).

\begin{figure}[tp]
\centerline{\includegraphics[width=0.5\textwidth]{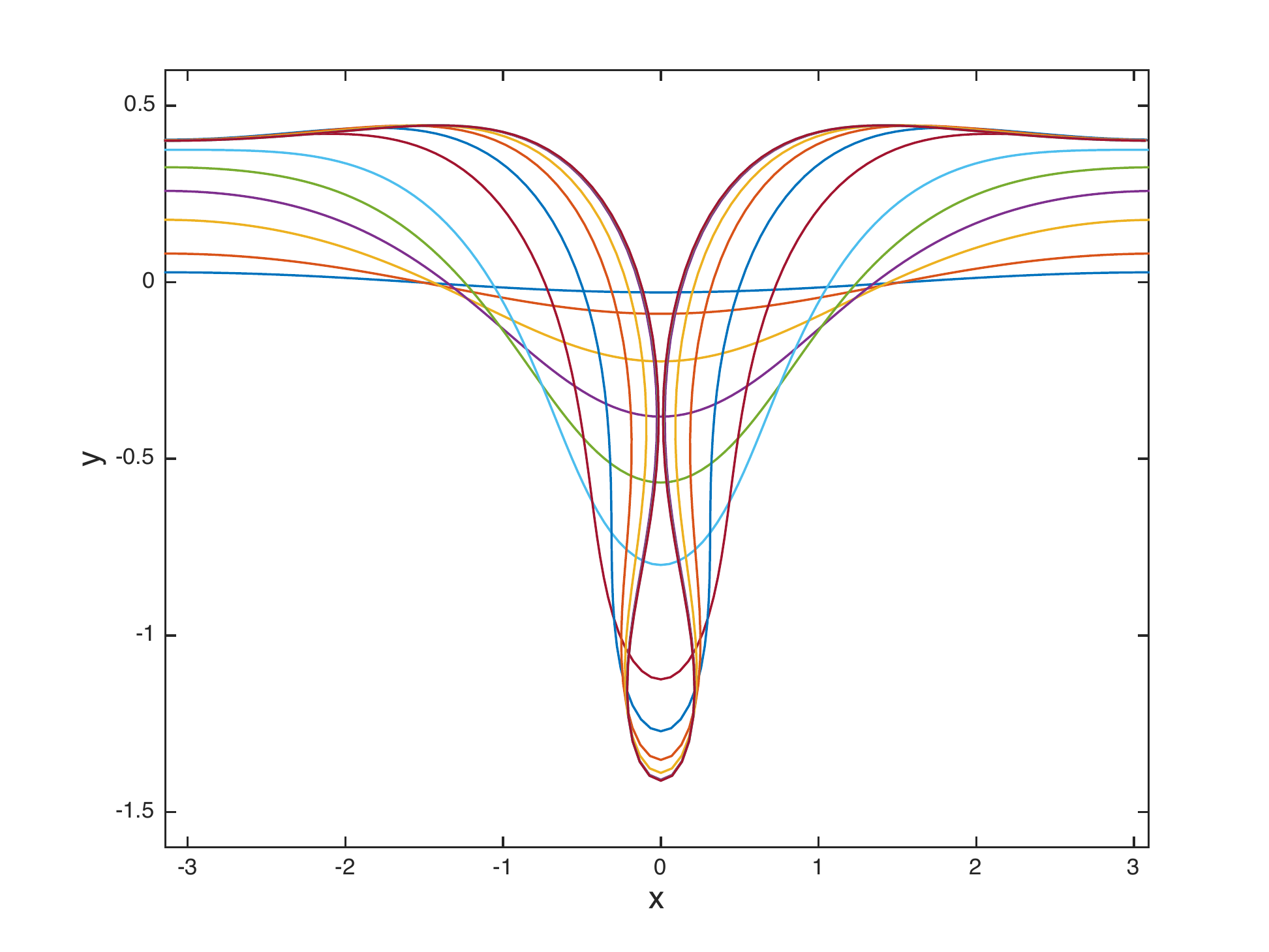}\includegraphics[width=0.5\textwidth]{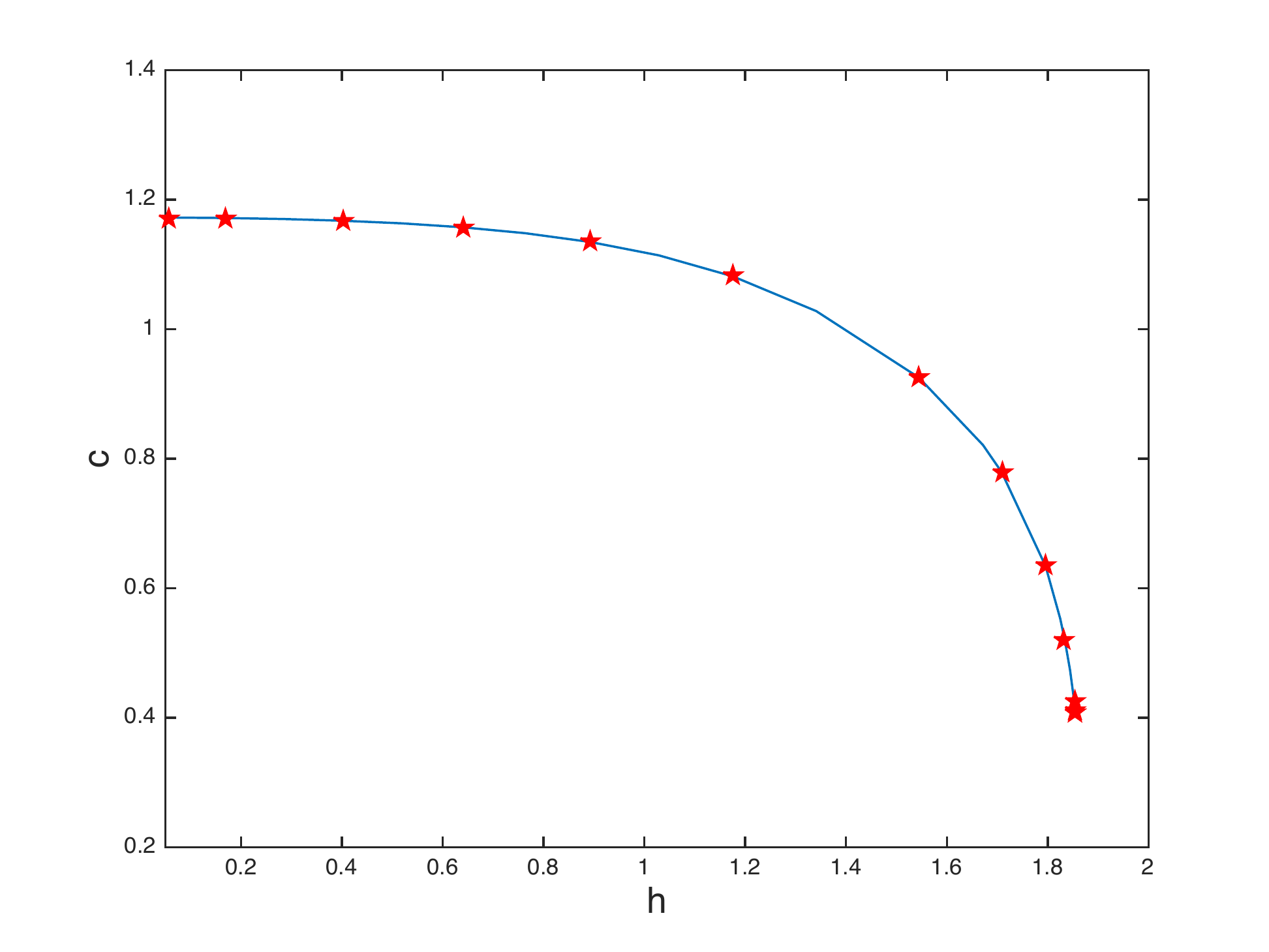}}
\caption{\it \small  An example computation of an entire branch of traveling waves, with $S=0.25$, $\tau=2$, $A=1$ and $\tilde{A}=0.2$.  A sampling of wave profiles at different locations on the branch are depicted in the left panel.  These profiles are marked with stars on the speed amplitude curve in the right panel.  The branch terminates with a self-intersecting profile. \label{BranchFig}}
\end{figure}

We proceed to numerically solve these $2N+1$ algebraic equations with
Broyden's method (a quasi-Newton method which approximates the Jacobian of
the system with a rank-one update to the Jacobian at the previous iteration; see
\cite{Broyden}). \ A small-amplitude solution to the
linearized equations (\ref{linearized_equations}) (with linear wave speed $%
c_{\pm }$ given by (\ref{c_pm})) is used as the initial guess for Broyden's
method, where amplitude (displacement) is specified as the $y$-coordinate of
the free-surface at the central node $x=0$. \ After iterating to a solution
within a desired tolerance, we record the solution. \ 

Then, as is typically done in these types of continuation methods, we look
for more solutions along the same branch by perturbing the previously computed
solution by a small amount in some direction, then using this perturbation as an initial guess for
the next application of Broyden's method.  The perturbation direction is called the continuation parameter. \  We begin using total displacement as the
continuation parameter. \ If a given ``step size" of displacement does not yield
convergence, then we choose a smaller (i.e. halve the previous step size)
perturbation from the last known solution as an initial guess. \ However, if
the step size drops below a given threshold, we switch to using a Fourier mode as our continuation parameter.  \ If the step size for this
continuation process becomes to small, we subsequently continue in higher
Fourier modes.

We follow a branch of solutions until the solution self-intersects (i.e.
outcome (d) of Theorem \ref{main_theorem}), returns to the trivial
solution (i.e. outcome (e)), or becomes to large to resolve (evidence of outcome (a)). \  After any such termination criterion is achieved, we
cease the continuation process, and record all solutions along the branch.  An example of a computed branch of waves which terminates in self-intersection is in Figure \ref{BranchFig}.

\subsection{Results}

\begin{figure}[tp]
\centerline{\includegraphics[width=0.5\textwidth]{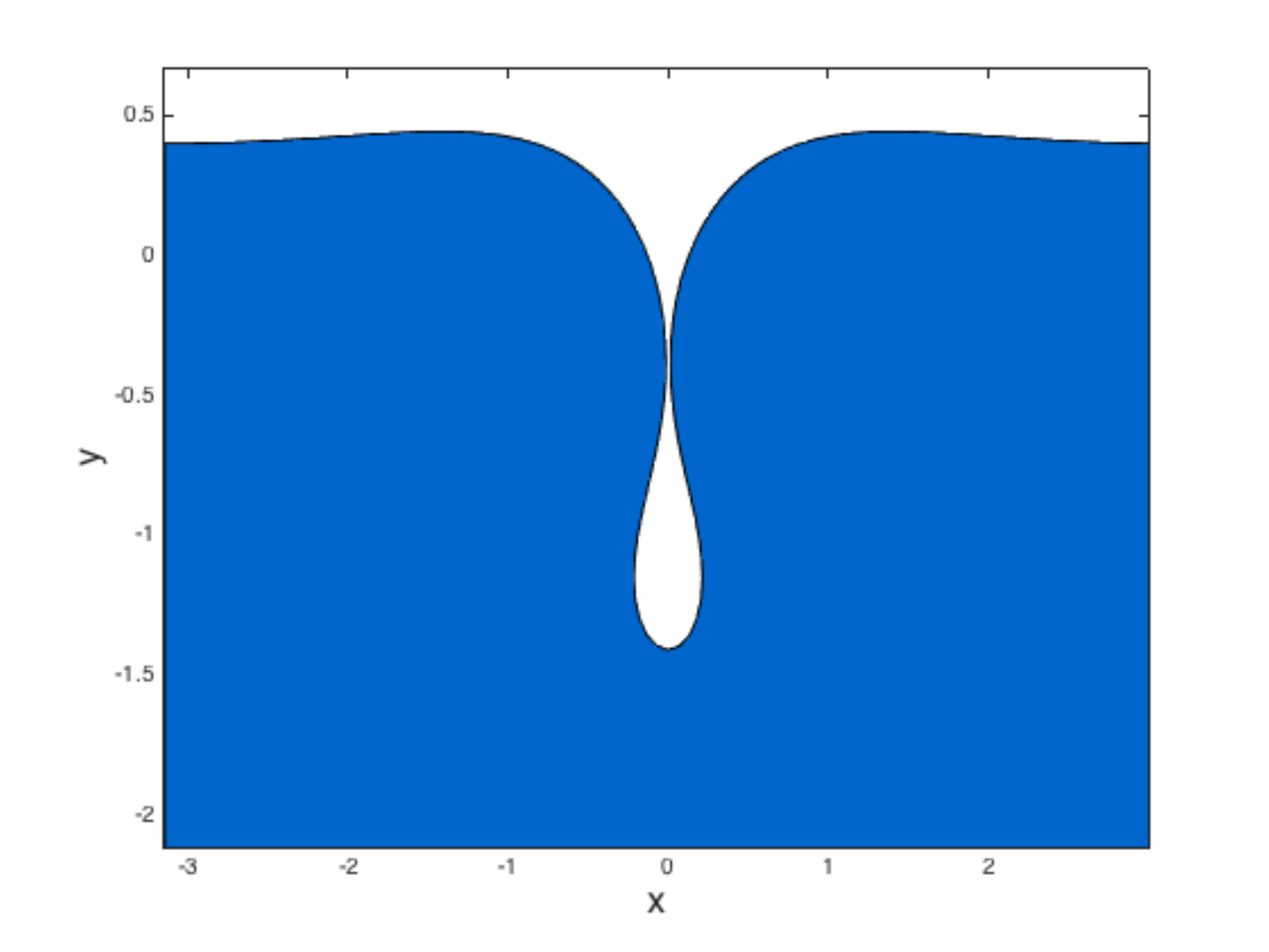}\includegraphics[width=0.5\textwidth]{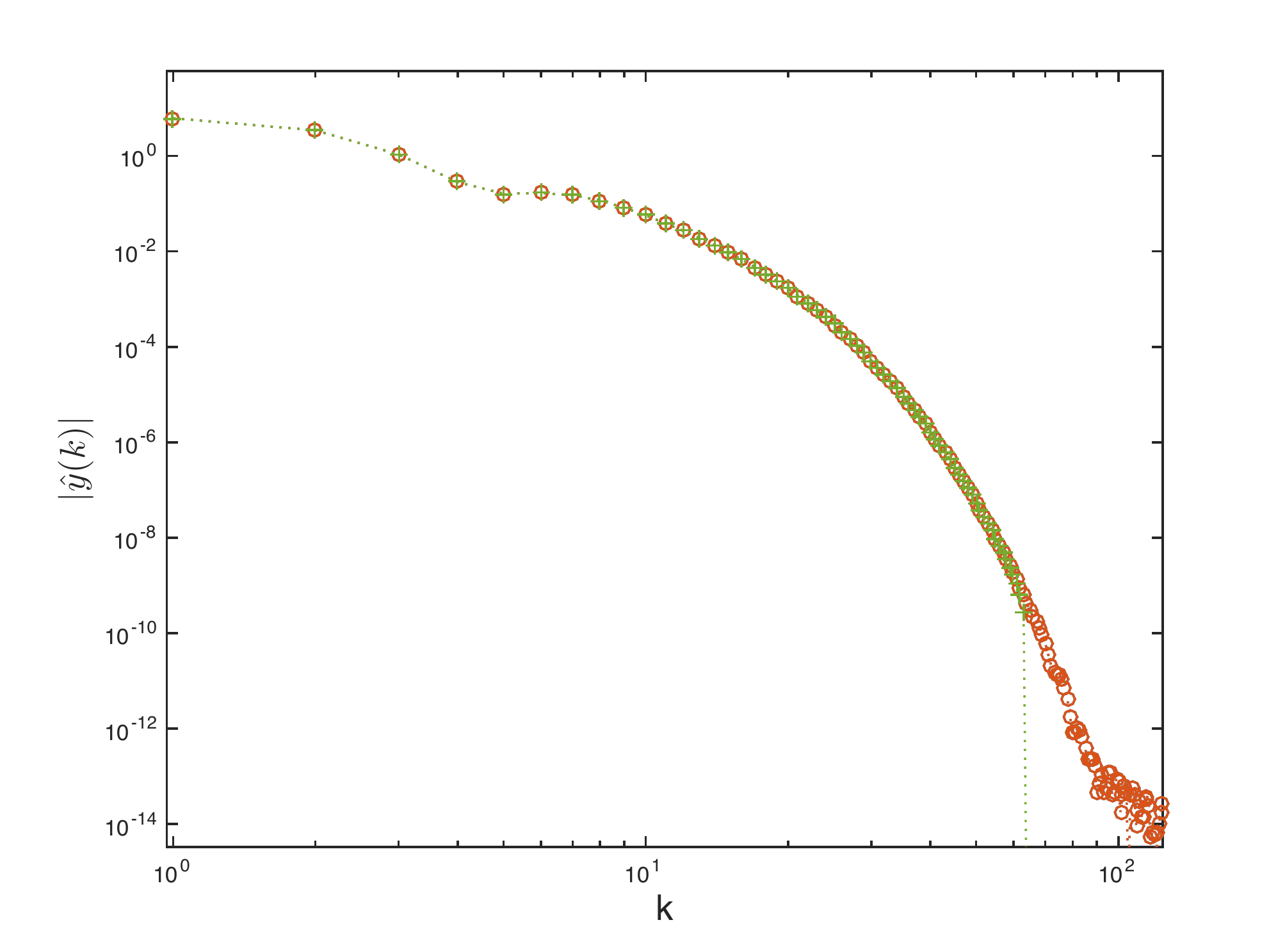}}
\caption{\it An example of a profile just before the self-intersecting configuration, for the branch of waves in Figure \ref{BranchFig}.  The wave is depicted in the left panel.  The right panel depicts the Fourier modes of the displacement, when $N=128$ points, marked with green plus signs, and $N=256$ points, marked with orange circles. \small \label{SelfFig}}
\end{figure}

In this section we present computations of global branches of traveling hydroelastic waves.  Global branches are computed which terminate in a self-intersecting profile, as well as branches whose most extreme representative is a standing wave, $c=0$.  We pay particular focus to the role of $\tilde{A}$, a proxy for the mass of the ice sheet.  This parameter was chosen as it appears only in the nonlinearity; the linear speeds and infinitesimal profiles are independent of $\tilde{A}$.

The quasi-Newton iteration described in the previous section uses an error threshold of $10^{-9}$, approximately the size of the floating point errors in approximating the derivatives in the Jacobian matrix of the quasi-Newton iteration via finite difference approximations.  The error is defined as the infinity norm of the Fourier modes of the projection of (\ref{tan_equation_for_numerical}) and (\ref{normal_equation_new2}). The bulk of the numerical results use $N=128$ points to discretize the interval of the pseudo-arclength $\alpha\in [0,2\pi)$.  When $N=128$ the most extreme waves computed have Fourier modes which decay to approximately this threshold, thus the choice of our discretization and error threshold are self-consistent.  As a check to see that the waves at this resolution are resolved to the reported threshold, we also computed a single, representative branch at $N=256$.  The waves profiles agree within the expected threshold; the Fourier modes of the extreme wave on this branch are reported at both resolutions are in the right panel of Figure \ref{SelfFig}.  The profile used for this comparison is reported in the left panel of Figure \ref{SelfFig}.

From the perspective of the global bifurcation theorem, self-intersecting waves result in a branch terminating at finite amplitude, case (d) of the theorem.  Standing waves signify a return to trivial, case (e) of the theorem.  At a standing wave a branch of waves with positive speed is connected to a branch of traveling waves with negative speed.  This setup can equally be interpreted as a branch which begins at one flat state configuration with one speed, and terminates at another 
flat configuration with a different speed.  

\begin{figure}[tp]
\centerline{\includegraphics[width=0.5\textwidth]{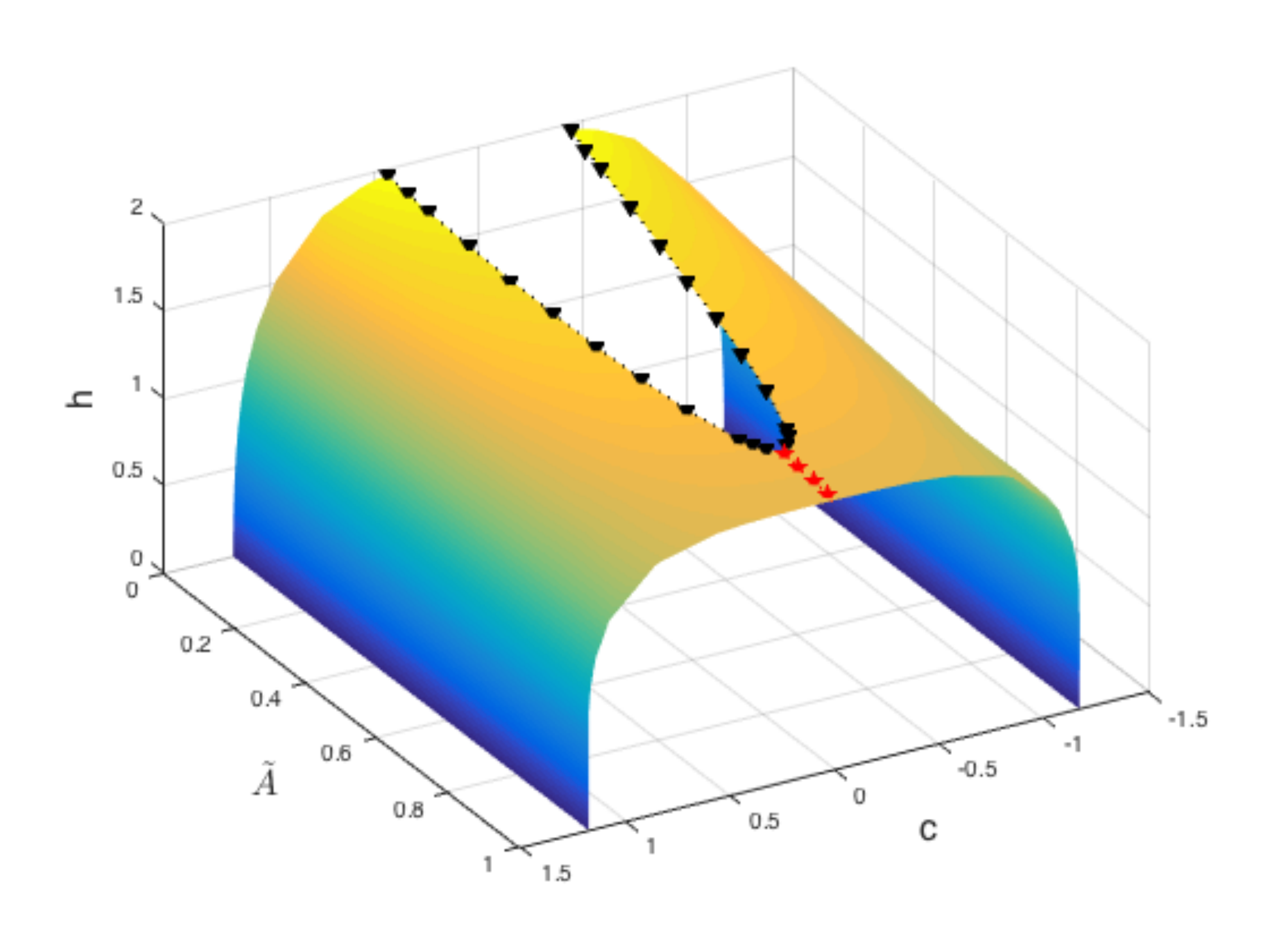}\includegraphics[width=0.5\textwidth]{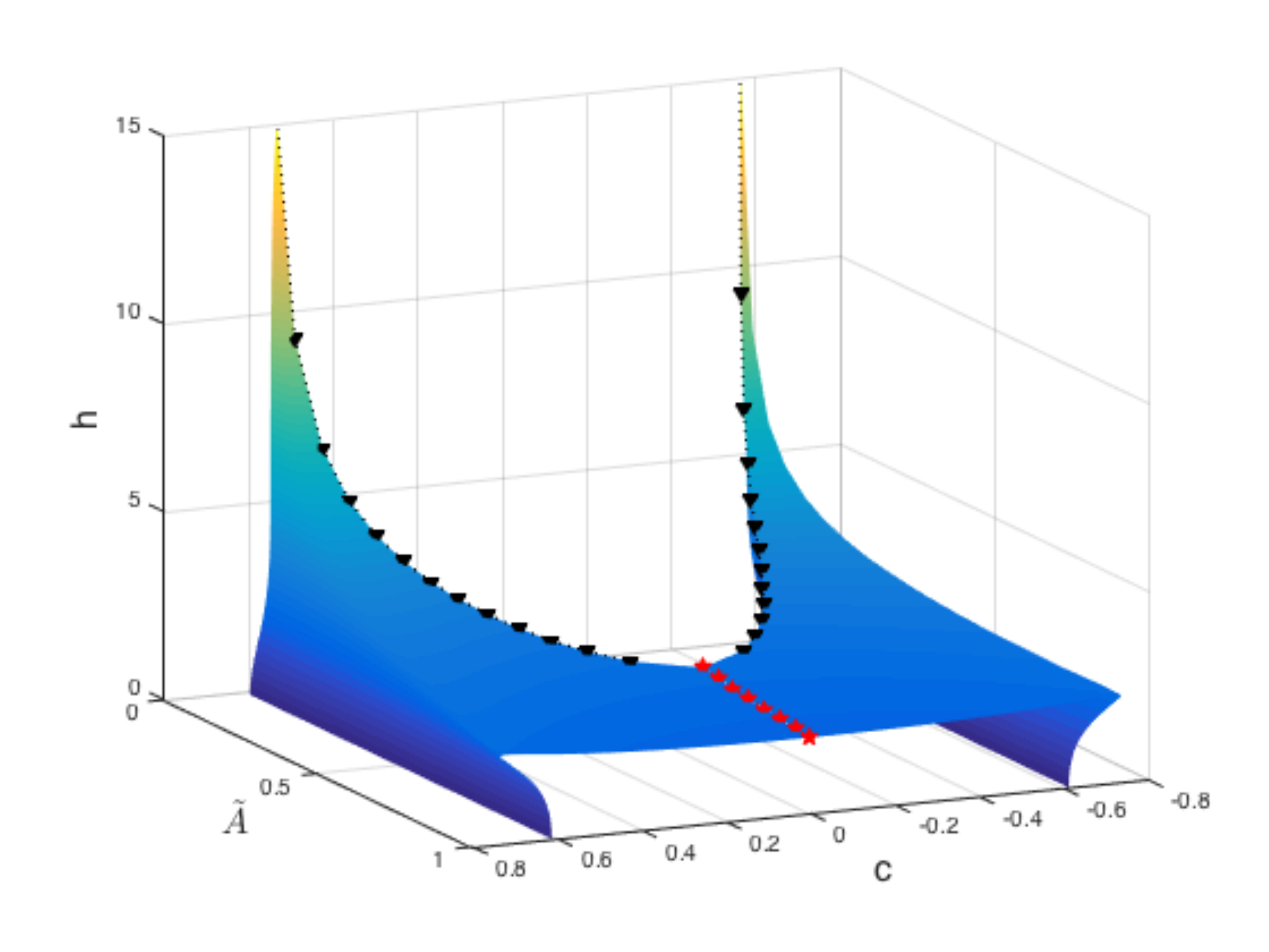}}
\caption{\it \small Two examples of bifurcation surfaces are depicted in the parameter space of $\tilde{A},c$, and the total interface displacement $h=\max(y)-\min(y)$. These waves were computed with $S=0.25$, $\tau=2$ and $A=1$ (left panel) and $A=0$ (right panel).  Branches with small $\tilde{A}$, in the back of the figure, terminate in self-intersecting waves, whose locations are marked with black triangles. In both cases, there is a critical $\tilde{A}$, corresponding to a switch from branches which terminate in self-intersecting waves to those which end in standing waves, whose locations are marked with red stars. The standing waves mark the merger of the surfaces of positive and negative speeds. \label{SurfaceFig}}
\end{figure}

We have numerically computed two examples of bifurcation surfaces, composed of a continuous family of branches of traveling waves with varying $\tilde{A}$.  These bifurcations surfaces are presented in Figure \ref{SurfaceFig}.  As extreme examples, we computed bifurcations surfaces with $A=0$, the density-matched case, and $A=1$ where the upper fluid is a vacuum.  Each wave on these surfaces has the same values of $S=0.25$ and $\tau=2$.   We present these surfaces in the three-dimensional space of $c,$ $\tilde{A},$ and total displacement $h=\max(y)-\min(y)$.  We chose to compute the surfaces for varying $\tilde{A}$, because the linear wave speeds don't depend on $\tilde{A}$.  Thus the changes in the surface for different $\tilde{A}$ are fundamentally due to large-amplitude, nonlinear effects.  In both computed surfaces, ($A=0$ and $A=1$), we observe that for small $\tilde{A}$, branches of traveling waves terminate in self intersection. After some critical $\tilde{A}$, the extreme wave on a branch is a standing wave; the branches of waves with positive speed are connected to branches of waves with negative speed, a ``return-to-trivial'' global bifurcation.  

In addition to computing bounded branches of traveling waves, we observe evidence of an unbounded branch.  In the case where the fluid densities match $A=0$ and the interface has no mass $\tilde{A}=0$, we found no evidence of a largest wave.  Considering the limit as $\tilde{A}\rightarrow 0$, we observe a largest self intersecting wave with total displacement $h\sim \tilde{A}^{-1/2}$, suggesting that wave when $\tilde{A}\rightarrow 0$, the interfaces can be unboundedly large.  This behavior is depicted both in the right panel of Figure \ref{SurfaceFig} as well as in the left panel of Figure \ref{DecayFig}.  

In search of an unbounded branch of traveling waves, we computed a branch of traveling waves in the configuration, $(\tilde{A},A,S,\tau)=(0,0,0.25,2)$.  The results of this computation are in the center and right panels of figure \ref{DecayFig}.   In the center panel are examples of increasingly large profiles of traveling waves.  We observe no evidence of a largest profile, or any tendency toward self-intersection.  The speed's dependence on displacement if depicted in the right panel of Figure \ref{DecayFig}.  The speed limits on a finite value, as the profiles become arbitrarily large.  We consider this configuration an example of case (a) of the main theorem.

\begin{figure}[tp]
 \centerline{\includegraphics[width=.375\textwidth]{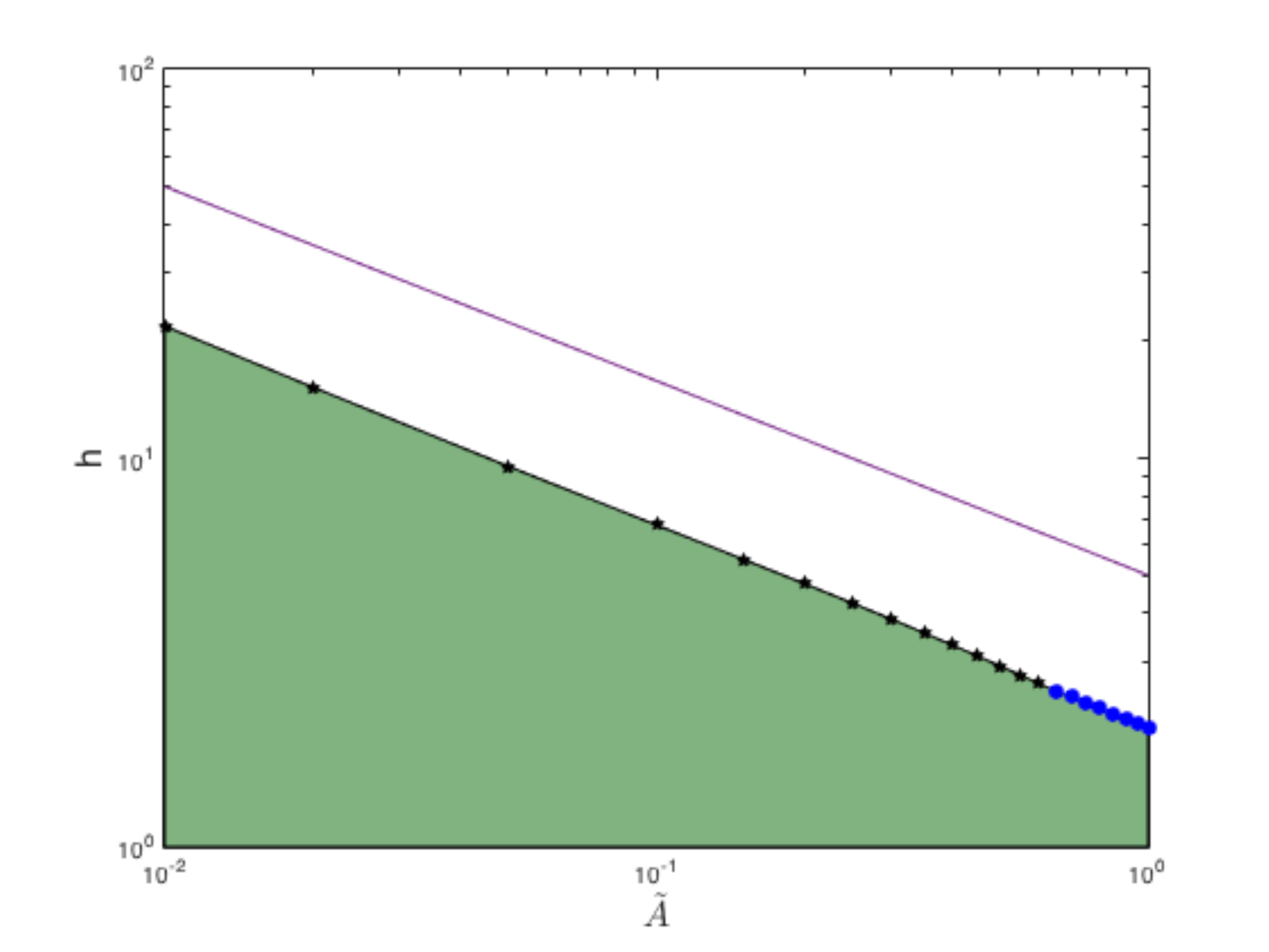}
\includegraphics[width=.375\textwidth]{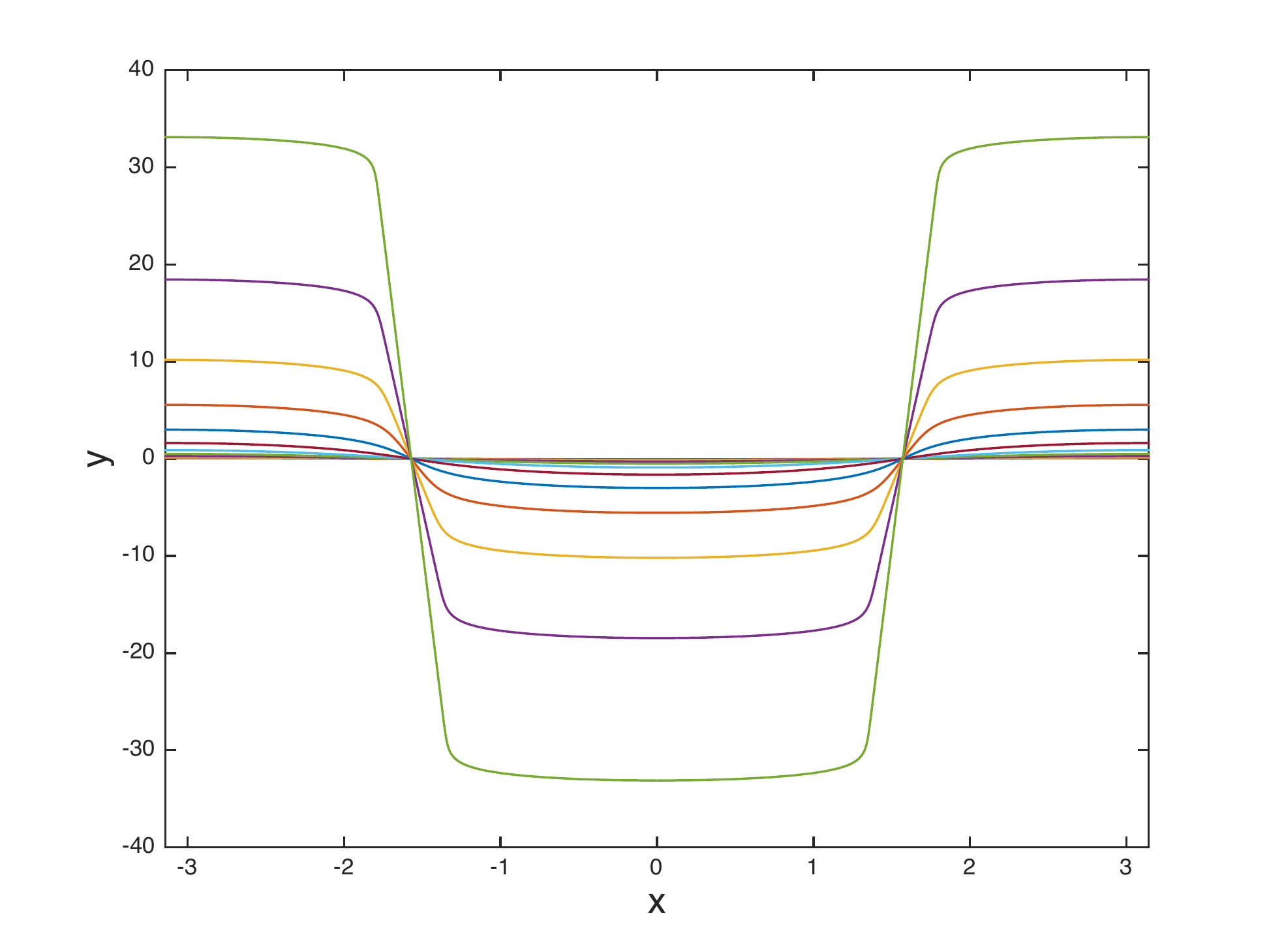}
\includegraphics[width=.375\textwidth]{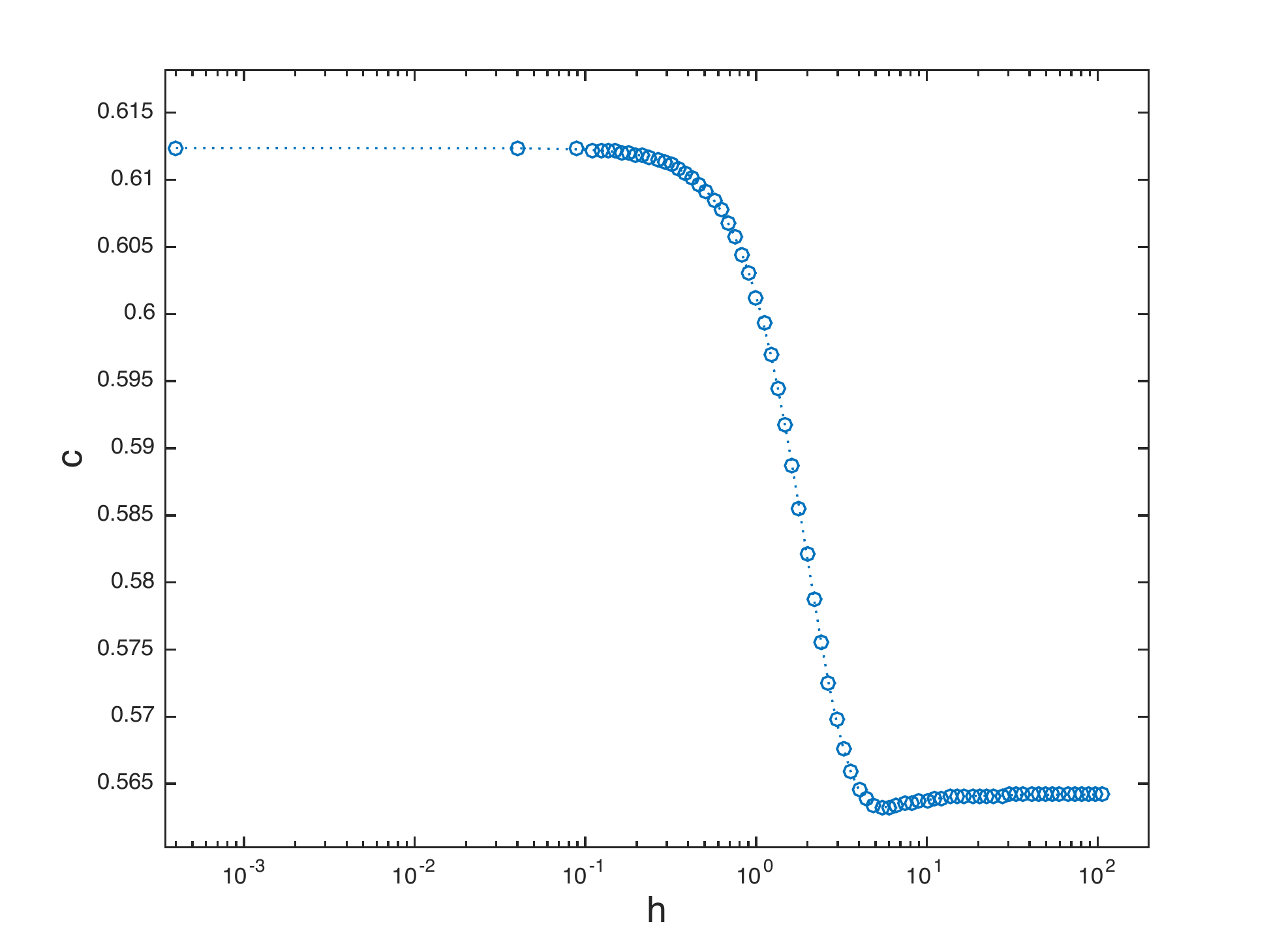}}
\caption{\it \small Numerical evidence of a branch of waves which contains interfaces of arbitrary length, case (a) of the main theorem. In the left panel, the displacement of the largest wave, $h$, for $A=0,S=0.25$, and $\tau=2$, is depicted as a function of $\tilde{A}$.  Numerical computations are marked, either self-intersecting waves or standing waves, using the marking convention of Figure \ref{SurfaceFig}.  As a guide, the value $h=5\tilde{A}^{-1/2}$ is marked with a solid line; waves exist in the shaded region. In the center are profiles computed at $(\tilde{A},A,S,\tau)=(0,0,0.25,2);$ no evidence of a largest profile was found.  On the right is the speed amplitude curve, with a logarithmic horizontal axis, for the same configuration as the center.  The speed limits on a finite value; thus, this is not case (f). \label{DecayFig}}
\end{figure}

\bibliography{references}{}
\bibliographystyle{plain}

\end{document}